\documentclass{amsart}
\usepackage{amssymb
,amsthm
,amsmath
,amscd
,mathtools,enumerate
}
\usepackage[shortlabels]{enumitem}
\usepackage[pagewise]{lineno}
\usepackage{hyperref}
\usepackage[all]{xy}
\usepackage[top=30truemm,bottom=30truemm,left=25truemm,right=25truemm]{geometry}

\newcommand{\Sp}{\mathrm{Sp}}
\newcommand{\Uu}{\mathrm{U}}
\newcommand{\SL}{\mathrm{SL}}
\newcommand{\GL}{\mathrm{GL}}

\newcommand{\Oo}{\mathrm{O}}
\newcommand{\SO}{\mathrm{SO}}
\newcommand{\GSp}{\mathrm{GSp}}
\newcommand{\GSpin}{\mathrm{GSpin}}
\newcommand{\GSO}{\mathrm{GSO}}
\newcommand{\PGL}{\mathrm{PGL}}

\newcommand{\GO}{\mathrm{GO}}

\newcommand{\Hom}{\mathrm{Hom}}
\newcommand{\Gal}{\mathrm{Gal}}

\newtheorem{thm}{Theorem}[section]

\newtheorem{lem}[thm]{Lemma}
\newtheorem{prop}[thm]{Proposition}
\newtheorem{coro}[thm]{Corollary}

\theoremstyle{remark}
\newtheorem{rem}[thm]{Remark}

\theoremstyle{definition}

\numberwithin{equation}{section}

\makeatletter
\def\iddots{\mathinner{\mkern1mu\raise\p@
	\hbox{.}\mkern2mu\raise4\p@\hbox{.}\mkern2mu
	\raise7\p@\vbox{\kern7\p@\hbox{.}}\mkern1mu}}
\def\adots{\mathinner{\mkern2mu\raise\p@\hbox{.}
 \mkern2mu\raise4\p@\hbox{.}\mkern1mu
 \raise7\p@\vbox{\kern7\p@\hbox{.}}\mkern1mu}}
\makeatother

\allowdisplaybreaks
\title{The Prasad conjectures for $\Uu_2$, $\SO_4$ and $\Sp_4$}
\author{Hengfei LU}
\address{Department of Mathematics, Weizmann Institute of Science, 234 Herzl St.  P.O.B. 26, Rehovot, 7610001, Israel}
\email{hengfei.lu@weizmann.ac.il}
\begin{document}
%\maketitle
\maketitle
\begin{abstract}
We will give a proof to the Prasad conjectures for  $\Uu_2$,  $\SO_4$ and $\Sp_4$ over a quadratic field extension. 
\end{abstract}
	\subsection*{Keywords}  distinction problems, base change map, the Prasad conjecture, Whittaker models 
\subsection*{MSC(2000)} 11F27$\cdot$11F70$\cdot$22E50
\tableofcontents
\section{Introduction}
Assume that $F$ is a nonarchimedean local field with characteristic $0$.
Let $G$ be a connected reductive group defined over $F$ and $H$
be a closed subgroup of $G$. Given a smooth irreducible representation $\pi$ of $G(F),$ one can study  the complex vector space $\Hom_{H(F)}(\pi,\mathbb{C}).$ If it is nonzero, then we say that $\pi$ is  $H(F)$-distinguished. 
\par
Period problems,
which are closely related to Harmonic Analysis, have been extensively studied for classical groups. 
 The most general situations have been studied  in  \cite{sakellaridis2012periods} when $G$ is split.
 Given a spherical variety $X=H\backslash G$, Sakellaridis-Venkatesh \cite{sakellaridis2012periods}  introduce a certain complex reductive group $\hat{G}_X$ associated with the variety $X$, to deal with  the spectral decomposition of $L^2(H\backslash G)$ under the assumption that $G$ is split.
 In a similar way, Dipendra Prasad \cite[\S9]{prasad2015arelative} introduces a certain quasi-split reductive group $H^{op}$ to deal with the distinction problem with respect to a quadratic character $\chi_H$ (depends on the quadratic extension $E/F$ and the reductive quasi-split group $H$), when the subgroup $H$ is the Galois fixed points of $G$, i.e. $H=G^{\Gal(E/F)}$, where $E$ is a quadratic field extension of $F$. The cases
 $H=\SL_2$ and $H=\SL_n$ ($n\geq3$) have been proved in \cite{anandavardhanan2003distinguished,anandavardhanan2016distinguished,L2018pacific} and the cases $H=\GL_2$ and $H=\PGL_2$ have been verified in \cite{dipendra1992invariant,lu2016new}.
 The cases $H=\GSp_4$ and $H=\SO_5$ have been studied for the tempered representations in \cite{lu2018GSp(4)}.
  In this paper,
  we will focus on the following cases: 
  \begin{itemize}
  	\item $G=R_{E/F}\Uu_2$ and $H=\Uu_2$,
  	\item $G=R_{E/F}\SO_4$ and $H=\SO_4$,
  	\item $G=R_{E/F}\Sp_4$ and $H=\Sp_4$,
  \end{itemize}
   where $R_{E/F}$ denotes the Weil restriction of scalars. Let $\theta$ be the involution defined on $G(F)$ and $H(F)=G^\theta$. More precisely, given $g\in G(F)$, $g^t$ denotes its transpose matrix,
  \[\theta(g)=\begin{cases}
  \omega_0\cdot\sigma((g^t)^{-1})\cdot\omega_0^{-1}&\mbox{ if }G(F)=\GL_2(E)\\
    \sigma(g)&\mbox{ if }G(F)=\SO_4(E) \mbox{  or  }\Sp_4(E)
  \end{cases} \]
  where $\sigma$ is the nontrivial element in $\Gal(E/F)$ and $\omega_0=\begin{pmatrix}
  &1\\-1&
  \end{pmatrix}$. So $H(F)=\{g\in G(F):\theta(g)=g \}$.
  
\par
Let $W_F$ (resp. $W_E$) be the Weil group of $F$ (resp. $E$) and let  $WD_F$ (resp. $WD_E$) be the Weil-Deligne group. Let $\psi$ be any nondegenerate additive character of $F$ and $\psi_E=\psi\circ tr_{E/F}.$
Assume that $\pi$ is an irreducible smooth representation of $H(F)$
 with a Langlands parameter %for $GL_2,$
$\phi_\pi:WD_F\longrightarrow {}^LH$ and a character $\lambda$ of the component group $S_{\phi_\pi }=C(\phi_\pi)/C^\circ(\phi_\pi),$ where $C(\phi_{\pi})$ is the centralizer of $\phi_{\pi}$ in $\hat{H}$ (the Langlands dual group of $H$) and $C^\circ(\phi_{\pi})$ is the connected component of $C(\phi_\pi).$
 Then $\phi_\pi|_{WD_E}$
gives a Langlands parameter of $H(E)=G(F)$. The map $$\Phi:\phi_{\pi}\longrightarrow \phi_{\pi}|_{WD_E}$$ is called the base change map. The parameter $\phi_{\pi}$ is called a lifted parameter or an extension of $\phi_\pi|_{WD_E}$. The Prasad conjecture for the quasi-split group $H$ implies the following:
%Then we can talk about the period problem.
%Assume $\pi$ is an admissible representation of $GL_2(E),$
%$\chi$ is a character of $GL_2(F),$ regarded as a subgroup of $GL_2(E).$
%$\pi$ has a nonzero $(GL_2(F),\chi)$-period if there is a nonzero intertwining operater from $\pi|_{GL_2(F)}$
%to $\chi.$
%If %$Hom_{GL_2(F)}(\pi,\mathbb{C})$ is nonvanishing, 
%$\chi$ is trivial, then we say that $\pi$
%has a nonzero $GL_2(F)$-period and $\pi$ is $GL_2(F)$-distinguished. 
\begin{thm}\label{localmain}
Let  $E$ be a quadratic field extension of a nonarchimedean local field $F$ with associated 
 Galois group $\Gal(E/F)=\{1,\sigma \}$. 
Assume that $\pi$ is an irreducible smooth admissible representation of $G(F)$ with an enhanced $L$-parameter $(\phi_\pi,\lambda)$ and that the $L$-packet $\Pi_{\phi_{\pi}}$ is generic. Then
\begin{enumerate}[(i)]
\item  If $\pi$ is  $H(F)$-distinguished, then $\pi^\vee \cong \pi^\theta $ 
%as an equality of $L$-packets
 and there exists a Langlands parameter $\phi$ of $H^{op}(F)$ such that $\phi|_{WD_E}=\phi_\pi$, where $H^{op}=\begin{cases}
\GL_2,&\mbox{if }H=\Uu_2;\\
\SO_4,&\mbox{if }H=\SO_4.
\end{cases}$
\item  If $\phi_\pi=\phi_{\pi'}|_{WD_E}$ for some irreducible representation $\pi'$ of $H^{op}(F)$ and
$\pi$ has a Whittaker model with respect to a non-trivial  additive character $\psi_0$ of $N(E)$ which is trivial on $N(F)$, where $N$ is a unipotent subgroup of $H$, then $\pi$ is $H(F)$-distinguished for $H=\Uu_{2}$ and $\SO_4$. 
\item Suppose $H=\Uu_{2}$ or $\SO_4$. Then there is an identity
\begin{equation}\label{conjidentity}
\dim\Hom_{H(F)}(\pi,\mathbb{C})+\dim\Hom_{H'(F)}(\pi,\mathbb{C})=\sum_{\phi\in F(\phi_\pi) } m(\lambda,\phi)\frac{\deg \Phi(\phi) }{d_0(\phi)}
\end{equation}
where
\begin{itemize}
	\item $H'$ is the unique nontrivial pure inner form of $H$ defined over $F$; 
	\item $F(\phi_\pi )=\{\phi:WD_F\longrightarrow {}^LH^{op}:\phi|_{WD_E}=\phi_\pi \}$ is the set of the lifted parameters;
	\item $m(\lambda,\phi)$ is the mulitplicity of the trivial character contained in the restricted character $\lambda|_{S_{\phi}}$;
	\item $\deg\Phi(\phi)$ is the degree of the base change map at $\phi$;
	\item $d_0(\phi)=|Coker\{S_{\phi}\longrightarrow S_{\phi_\pi}^{\Gal(E/F)} \}|$ is the size of the cokernel. 
\end{itemize}
\end{enumerate}
\end{thm}
\begin{thm}\label{sp(4)}
	Let $\pi$ be an irreducible tempered  representation of $\Sp_4(E)$ with an enhanced $L$-parameter $(\phi_\pi,\lambda )$, distinguished by $\Sp_4(F)$.
	 Then the multiplicity
	\[\dim\Hom_{\Sp_4(F)}(\pi,\mathbb{C})=\sum_{\phi\in F(\phi_\pi) } m(\lambda,\phi)\frac{\deg \Phi(\phi) }{d_0(\phi)}  \]
	where $F(\phi_{\pi})$, $m(\lambda,\phi)$ and $d_0(\phi)$ are defined as above in \eqref{conjidentity}.
	%for a tempered representation $\pi$ of $\Sp_4(E)$.
\end{thm}
\par
Anandavardhanan-Prasad \cite{anandavardhanan2003distinguished,anandavardhanan2016distinguished} use the restriction to $\SL_n(E)$ from $\GL_n(F)$-distinguished representation of $\GL_n(E)$ to show that the mulitiplicity $\dim\Hom_{\SL_n(F)}(\pi,\mathbb{C})$ is equal to the number of inequivalent lifts where $H=\SL_n$ and $H^{op}=\mathrm{SU}_n$ with $n\geq3$ (where $H^{op}=\SL_2$ if $H=\SL_2$). We will use a similar method to deal with the case when $H=\SO_4$, which involves the distinction problems for $\GSO_4$ or $\GSpin_4$ over a quadratic field extension $E/F$. The main task in this paper is to verify the identity \eqref{conjidentity}. In particular, on the right hand side (called the Galois side) of \eqref{conjidentity}, we will construct and study the finite set $F(\phi_{\pi})$ case by case in \S3.
\par
 Thanks to the results 
\cite[Theorem 0.2]{lapid2012unitary} when $H=\Uu_2$, one can get the multiplicities on the left hand side (called the automorphic side) of \eqref{conjidentity} for almost all cases except  the Langlands parameter $\phi_\pi=2\chi$ where $\chi$ is Galois invariant. It turns out that the exception case for $\Uu_2$ can be 
obtained via the Prasad conjecture for $\SL_2$ over a quadratic field extension $E/F$. (See Lemma \ref{SL(2)}.)  When $H=\Sp_4$, we need to assume that $\pi$ is tempered because  the identity \eqref{conjidentity} for $H=\GSp_4$ has been discussed in \cite{lu2018GSp(4)} only when 
the representation of $\GSp_4(E)$ is tempered. Another difference for the distinction problems between $\SO_4$ and $\Sp_4$ is that the pair $(\GSO_4(E),\GSO_4(F))$ is a Gelfand pair while $(\GSp_4(E),\GSp_4(F))$ is not a Gelfand pair.
\begin{rem}
	In \cite[Theorem 1]{beuzart2017distinguished}, Raphael Beuzart-Plessis uses the relative trace formula to show that $$\dim\Hom_{H'(F)}(\pi',\chi_{H'})=\dim\Hom_{H(F)}(\pi,\chi_H),$$ where $H'$ is an inner form of $H$ defined over $F$, $\chi_{H'}$ is a quadratic character of $H'(F)$ and $\pi'$ is a stable square-integrable representation of $(R_{E/F}H')(F)=H'(E)$ which has a transfer to $\pi$ of $H(E)$. 
\end{rem}
\begin{rem}
	It has been shown that $\deg\Phi(\phi)=1$ if $\phi$ is discrete, and so is $d_0(\phi)$. (See \cite[Remark 22]{prasad2015arelative}.) 
%	In \cite{anan2006integral,anan2013localglobal}, Anandavardhanan and Prasad discuss the global period problems for $\SL_2$ over a quadratic number field extension $\mathbb{E}/\mathbb{F}$. More general, there are several results for the global period problems of $\SL_1(D)$ in \cite[\S9]{anan2013localglobal}, where $\SL_1(D)$ is an inner form of $\SL_2$ defined over a number field $\mathbb{F}$. We hope that we can also use the global theta correspondence to revisit these questions in future.
\end{rem}
Now we briefly describe the contents and the organization of this paper.
 The proof of Theorem \ref{localmain} for the quasi-split group $\Uu_{2}$ will be given in \S $2$. 
 In \S $3$, we study the distinction problems for $\GSO_4$ and $\SO_4$, including the cases when the group $H$ is not split but quasi-split.
The last section focuses on  the proof of Theorem \ref{sp(4)} which depends on the local Langlands correspondence for $\Sp_4$ proved by Gan-Takeda in \cite{takeda2010Sp(4)}.
\subsection*{Acknowledgments} The author thanks  Dipendra Prasad for his guidance and numerous discussions when he was a Visiting Fellow at Tata Insititute of Fundamental Research, Mumbai. He also thanks Yuanqing Cai, Wee Teck Gan and Sandeep Varma for helpful discussions. Part of the work was  done and was partially supported while he was visiting Wuhan University, Zhejiang University and the Institute for Mathematical Sciences of National University of Singapore in 2018. He would like to thank them for  their hospitalities.   This work was partially supported by the ERC, StG grant number 637912 during revision.
% He also thanks the anonymous referees for the careful reading and helpful comments, especially for pointing out the inaccurate statement in Theorem \ref{localmain} in the earlier version. 

\section{The Prasad conjecture for $\Uu_2$}
In this section, we will verify the Prasad conjecture for $H=\Uu_2$. In this case $G(F)=\GL_2(E)$ and $H'(F)$ is the unitary group of a $2$-dimensional Hermitian vector space $V$ with nontrivial discriminant, which is not quasi-split.
\par 
 First, let us recall \cite[Theorem 0.2]{lapid2012unitary}.
\begin{thm}[Feigon-Lapid-Offen]\label{lapidoffen}
	Suppose that $\pi$ is an irreducible admissible generic representation of $G(F)=\GL_2(E)$. If $\pi$ is $H(F)$-distinguished, then $\pi$ is Galois invariant. Conversely, assume $\pi=\pi^\sigma$. Then 
	\begin{itemize}
		\item If $\pi$ is a square-integrable representation, then $$\dim \Hom_{H(F)}(\pi,\mathbb{C} )=1=\dim\Hom_{H'(F)}(\pi,\mathbb{C}).$$
		\item If $\pi=\pi(\chi_1,\chi_2)$ with $\chi_1=\chi_2^\sigma\neq\chi_2$ is a principal series representation, then
		\[\dim\Hom_{H(F)}(\pi,\mathbb{C})=1\mbox{    and    }\dim\Hom_{H'(F)}(\pi,\mathbb{C})=0. \]
		\item If $\pi=\pi(\chi_1,\chi_2)$  with $\chi_1\neq\chi_2$ and $\chi_i=\chi_i^\sigma$ for $i=1,2$, then \[\dim\Hom_{H(F)}(\pi,\mathbb{C})=2=\dim\Hom_{H'(F)}(\pi,\mathbb{C}). \]
		\item If $\pi=\pi(\chi,\chi)$ with $\chi=\chi^\sigma$, then $\dim\Hom_{H(F)}(\pi,\mathbb{C}) \geq2$.
	\end{itemize}
\end{thm}
This section focuses on the last case $\phi_{\pi}=2\chi$ with $\chi=\chi^\sigma$. 
\begin{lem}\label{SL(2)}
	Assume that $\pi=\pi(\chi,\chi)$ is a representation of $\GL_2(E)$ with $\chi=\chi^\sigma$. Then
	\[\dim\Hom_{H(F)}(\pi,\mathbb{C} )=2=\dim\Hom_{H'(F)}(\pi,\mathbb{C}) .\]
\end{lem}
\begin{proof}
	By the Geometric Lemma of Bernstein-Zelevinsky, one can easily obtain that $$\dim\Hom_{H'(F)}(\pi(\chi,\chi),\mathbb{C})=2$$ since there are two orbits in the double coset $B(E)\backslash\GL_2(E)/H'(F)$ and both are open. For the quasi-split group $H(F)$, it suffices to show that $\dim\Hom_{H(F)}(\pi,\mathbb{C})\leq2 $. Together with Theorem \ref{lapidoffen}, one has $\dim\Hom_{H(F)}(\pi,\mathbb{C})=2$. Note that $\SL_2(F)$ is a subgroup of $\mathrm{U}_{2}(F)$ and so the multiplicity
	$\dim\Hom_{H(F)}(\pi,\mathbb{C})$
	 has an upper bound $\dim \Hom_{\SL_2(F)}(\pi,\mathbb{C})$, which is $2$ due to \cite[Theorem 1.3]{anandavardhanan2003distinguished}. Then we are done.
\end{proof}
Now we start to  give a proof to Theorem \ref{localmain} for $\Uu_2$.
\begin{proof}[Proof of Theorem \ref{localmain} for $\Uu_2$] Recall that $H(F)=\Uu_2$ and $\chi_H=\mathbf{1}$.
\begin{enumerate}[(i)]
	\item It is obvious that $\pi$ is $H(F)$-distinguished if and only if $\pi=\pi^\sigma$ due to Theorem \ref{lapidoffen}. If $\pi=\pi^\sigma$, then there is an extension $\tilde{\phi}:WD_F\longrightarrow\GL_2(\mathbb{C})$ of $\phi_{\pi}$.
	\item It follows from Theorem \ref{lapidoffen} directly.
	\item Assume $\pi=\pi^\sigma$. Otherwise, both sides of \eqref{conjidentity} are zero. In this case, $d_0(\phi)=1=m(\lambda,\phi)$. There are several cases:
	\begin{itemize}
		\item If $\phi_{\pi}$ is irreducible, i.e., $\pi$ is a square-integrable representation of $\GL_2(E)$, then there exist two lifts $\phi$ and $\phi\omega_{E/F}$, where $\omega_{E/F}$ is the quadratic character associated to the quadratic field extension $E/F$ by the local class field theory. They are in the same orbit under  twisting by $\omega_{E/F}$, so it corresponds to two inner forms.
		\item If $\phi_{\pi}=\chi_1+\chi_2$ with $\chi_1=\chi_2^\sigma\neq\chi_2$, then $\dim\Hom_{H'(F)}(\pi,\mathbb{C})=0 $. On the Galois side, there is only one lift
		\[\phi=Ind_{W_E}^{W_F}\chi_1 \]
		which corresponds to the trivial inner form in $H^1(F,\Uu_2)$.
		\item If $\phi_{\pi}=\chi_1+\chi_2$ with $\chi_i=\chi_i^\sigma$ and $\chi_1\neq\chi_2$, then there are four lifts and there are two orbits in $F(\phi_{\pi})$ under  twisting  by $\omega_{E/F}$. On the automorphic side, both the muliplicities $\dim\Hom_{\Uu_2}(\pi,\mathbb{C}) $ and $\dim\Hom_{H'(F)}(\pi,\mathbb{C})$ are $2$.
		\item If $\phi_{\pi}=2\chi$ with $\chi=\chi^\sigma=\chi_F|_{W_E}$, then there are three lifts $\phi_1=2\chi_F$, $\phi_2=\chi_F+\chi_F\omega_{E/F}$ and $\phi_3=\phi_1\otimes\omega_{E/F} $. Moreover, $\deg\Phi(\phi_1)=1=\deg\Phi(\phi_3)$ and $\deg\Phi(\phi_2)=2$.
	\end{itemize}
Therefore, the identity \eqref{conjidentity} holds for any generic representation $\pi$ of $\GL_2(E)$.
\end{enumerate}
\end{proof}

\section{The Prasad conjecture for $\SO_4$ }
In this section, we will investigate the Prasad conjecture for $\SO_4$. There are several cases for $H(F)$: the split group $\SO_{2,2}(F)$  and the quasi-split group $\SO(V)$ where  $V$ is the $4$-dimensional $F$-vector space with nontrivial discriminant corresponding to $K$ (a quadratic field extension of $F$) and Hasse invariant $+1$. (Sometimes we say that the discriminant of $V$ is $K$.) Moreover,
\[\SO(V)=\GL_2(K)^\natural/F^\times \]
where $\GL_2(K)^\natural=\{g\in\GL_2(K):\det(g)\in F^\times \}\cong\GSpin(V)$ and
 there is an isomorphism 
\[\GSO(V)\cong\frac{\GL_2(K)\times F^\times }{\triangle K^\times } \]
where $\triangle K^\times\hookrightarrow\GL_2(K)\times F^\times $ via $k\mapsto(k,N_{K/F}(k)^{-1})$.
The similitude character $\lambda_V:\GSO(V)\rightarrow F^\times$ is given by
\[\lambda_V:(g,t)\mapsto N_{K/F}(\det g)\cdot t^2. \]
\subsection{The Prasad conjecture for $\GSO(V)$}
Let us recall the fact when $H=\GSO(V)$, $\chi_H=\omega_{E/F}$
and
\[H^{op}(F)\cong\frac{\Uu_2(EK/K)\times\Uu_1(E/F) }{\Uu_1(EK/K)} \]
where $EK$ is the composite field containing two distinct quadratic fields $E$ and $K$ of $F$ with Weil group $W_{EK}$ and $\Uu_i(EK/K)$ is the unitary group of $i$-dimensional Hermitian $EK$-vector space with trivial discriminant.
\begin{thm}
	Let $K$ be a quadratic field extension of $F$, different from $E$. Consider $$\GSO(V)\cong\frac{\GL_2(K)\times F^\times }{\triangle K^\times} .$$ Let $\Sigma=\pi\boxtimes\chi_E$ be an irreducible generic admissible representation of $\GSO(V\otimes_F E)$ with $\omega_{\pi}=\chi_E\circ N_{EK/E}$ and $\chi_E|_{F^\times}=\mathbf{1}$. Then  $\chi_E$ is a character of $\Uu_1(E/F)$ and
	\begin{enumerate}[(i)]
		\item $\Sigma$ is $(H(F),\chi_H)$-distinguished if and only if $\pi$ is $(\GL_2(K),\omega_{E/F}\circ N_{K/F}\circ\det )$-distinguished;
		\item there exists an identity
		\[\dim\Hom_{\GSO(V)}(\Sigma,\omega_{E/F}\circ\lambda_V)=\dim\Hom_{\GL_2(K)}(\pi,\omega_{E/F}\circ N_{K/F}\circ\det)=| F(\phi_{\Sigma})|  \]
		where $F(\phi_{\Sigma})=\{\phi:WD_F\longrightarrow{}^LH^{op}:\phi|_{WD_E}=\phi_{\Sigma} \}$.
	\end{enumerate}
\end{thm}
\begin{proof}
	\begin{enumerate}[(i)]
		\item It follows from the diagram
		\[\xymatrix{1\ar[r]&\GL_1(EK)\ar[r]&\GL_2(EK)\times\GL_1(E)\ar[r]&\GSO(V\otimes_F E)\ar[r]&1 \\
		1\ar[r]&\GL_1(K)\ar[u]\ar[r]^-\triangle&\GL_2(K)\times\GL_1(F)\ar[u]\ar[r]&\GSO(V)\ar[u]\ar[r]&1 } \]
	and the assumption $\chi_E|_{F^\times}=\mathbf{1}$. Here $\triangle(k)=(k,N_{K/F}(k)^{-1})$.
	\item Note that $\pi$ is $(\GL_2(K),\omega_{E/F}\circ N_{K/F}\circ\det )$-distinguished if and only if there exists a representation $\rho$ of $\Uu_2(EK/K)$ such that the $L$-parameter $\phi_\rho $ satisfying $$\phi_{\rho}|_{WD_{E}}=\phi_{\pi}:WD_E\longrightarrow(\GL_2(\mathbb{C})\times\GL_2(\mathbb{C}))\rtimes\Gal(EK/E) .$$
	 Then $\rho\boxtimes\chi_E$ is an irreducible representation of $H^{op}(F)$ and $\phi_{\rho\boxtimes\chi_E}|_{WD_E}=\phi_{\Sigma}$.
	\end{enumerate}
This finishes the proof.
\end{proof}

If $K=E$, then $\GSO(V\otimes_F E)=\GSO_{2,2}(E)$. Let $\Sigma=\pi_1\boxtimes\pi_2$ be an irreducible generic representation of $\GSO_{2,2}(E)$ with $\omega_{\pi_1}=\omega_{\pi_2}$. Then $$\dim\Hom_{\GSO(V)}(\Sigma,\mathbb{C})=\begin{cases}
1,&\mbox{  if  } \pi_1^\vee=\pi_2^\sigma;\\
0,&\mbox{  otherwise.}
\end{cases} $$ 

\subsection{The Prasad conjecture for $\SO(V)$} Consider $$\SO(V)=\GL_2(K)^\natural/F^\times\cong\GSpin(V)/F^\times$$ and $K\neq E$. Let $\pi$ be an irreducible representation of $\SO(V\otimes_F E )$ with trivial central character. There exists a representation $\tau$ of $\GL_2(EK)^\natural$ with trivial central character such that
\[\dim\Hom_{\SO(V)}(\pi,\mathbb{C} )=\dim\Hom_{\GSpin(V)}(\tau,\mathbb{C}), \]
where $\GL_2(EK)^\natural=\{g\in\GL_2(EK):\det(g)\in E^\times \}$. Let $\Gal(E/F)=\langle\sigma\rangle=\Gal(EK/K)$ and $\Gal(K/F)=\langle s\rangle=\Gal(EK/E)$. Note that
\[\xymatrix{\GL_2(EK)^\natural\ar[r]^-\cong& \GSpin(V\otimes_F E)\ar@{^{(}->}[r]& (R_{K/F}\GL_2)(E)\\ \GL_2(K)^\natural\ar[r]^-\cong& \GSpin(V)\ar[u]\ar@{^{(}->}[r]&(R_{K/F}\GL_2)(F).\ar[u] } \]
Assume that $\tau$ is $\GSpin(V)$-distinguished. Then there exists a representation $\tilde{\tau} $ of $\GL_2(EK)$ distinguished by $\GL_2(K)$. Moreover, the central character $\omega_{\tilde{\tau}}=\mu^s\mu^{-1}$ with $\mu(-1)=1$ for   a character $\mu$ of $\GL_1(EK)$ and $\omega_{\tilde{\tau}}|_{K^\times}=\mathbf{1}$. 
%So $\omega_{\tilde{\tau}}=\mathbf{1}$ if $p\neq2$.
Thanks to \cite[Theorem 1.2]{anandavardhanan2003distinguished}, the finite sets $$Y_{\tilde{\tau} }=\{\chi:(EK)^\times/E^\times\longrightarrow\mathbb{C}^\times|\tilde{\tau}\otimes \chi\cong\tilde{\tau} \},$$
$Z_{\tilde{\tau}}=\{\chi\in Y_{\tilde{\tau}}|\chi|_{K^\times}=\mathbf{1}  \} $ and
\[X_{\tilde{\tau}}= \{\chi:K^\times/F^\times\longrightarrow\mathbb{C}^\times| \Hom_{\GL_2(K) }(\tilde{\tau},\chi\circ\det )\neq0  \}  \]
play a vital role in computing the multiplicity $\dim\Hom_{\GSpin(V)}(\tau,\mathbb{C})=\dim\Hom_{\SO(V)}(\pi,\mathbb{C})  $.
\begin{lem}\cite[Proposition 4.2]{anandavardhanan2003distinguished}
	Assume that $\tilde{\tau}$ is a square-integrable representation of $\GL_2(EK)$, distinguished by $\GL_2(K)$. Then there exists a bijection between the set $X_{\tilde{\tau}}$ and $Z_{\tilde{\tau}}$.
\end{lem}
\begin{proof}
		Suppose that $\chi\in X_{\tilde{\tau}}$. Then there exists a character $\mu$ of $\GL_1(EK)$ such that  $$\Hom_{\GL_2(K)}(\tilde{\tau}, \chi\circ\det )=  \Hom_{\GL_2(K) }(\tilde{\tau}\otimes\mu,\mathbb{C})\neq0$$  and $\mu|_{K^\times}=\chi^{-1}$. So 
		\[(\tilde{\tau}\otimes\mu )^\vee\cong(\tilde{\tau}\otimes\mu )^\sigma \]
		where $\tilde{\tau}^\vee$ denotes the contragredient representation of $\tilde{\tau}$.
		Note that $\tilde{\tau}$ is $\GL_2(K)$-distinguished. Then $\tilde{\tau}^\sigma\cong\tilde{\tau}^\vee$. Therefore $\tilde{\tau}^\vee\otimes\mu^{-1}\cong\tilde{\tau}^\sigma\otimes\mu^\sigma\cong\tilde{\tau}^\vee\otimes\mu^\sigma$ and so $$\tilde{\tau}\otimes\chi\circ N_{EK/K}\cong\tilde{\tau}\otimes(\mu\mu^\sigma)^{-1}\cong\tilde{\tau}.$$
	Thus $\chi\circ N_{EK/K}\in Y_{\tilde{\tau}} $.
	Observe that $\chi\omega_{EK/K}\notin X_{\tilde{\tau}}$ because $\tilde{\tau}$ is square-integrable, which implies that $$\Hom_{\GL_2(K)}(\tilde{\tau}\otimes\mu,\omega_{EK/K}\circ\det )=0.$$
	Moreover, $\chi\circ N_{EK/K}|_{K^\times}=\chi^2=\mathbf{1}$. Hence $\chi\mapsto \chi\circ N_{EK/K}$ gives an injective map from $X_{\tilde{\tau}}$ to $Z_{\tilde{\tau}}$.
	\par
	 Conversely, if $\chi\in Z_{\tilde{\tau}}$, then $\chi^2=\mathbf{1}$ and $\chi^s=\chi=\chi^\sigma$. There exists a character $\chi_F$ of $F^\times$ such that $\chi=\chi_F\circ N_{EK/F}$. Note that $\chi|_{K^\times}=\mathbf{1}=\chi|_{E^\times}$. Then $\chi_F^2=\mathbf{1}$.
	Set $\chi_K=\chi_F\circ N_{K/F}$. Then
	\[\chi_K|_{F^\times}=\chi_F^2=\mathbf{1}. \]
%	\[\chi_K(tt^st^\sigma t^{t\sigma})=\chi_K\circ N_{EK/K}(tt^s)=\chi(t)\chi^s(t)=1 \]
%	for any $t\in\GL_1(EK)$. So $\chi_K|_{F^\times}=1$ since the norm map from $EK$ to $F$ is surjective. 
Suppose that $\chi=\mu\mu^\sigma$ with $\chi_K=\mu|_{K^\times}$ for a character $\mu$ of $\GL_1(EK)$. Then 
	\[(\tilde{\tau}\otimes\mu)^\vee\cong(\tilde{\tau}\otimes\chi\otimes\mu)^\vee\cong\tilde{\tau}^\vee\otimes \chi\otimes\mu^{-1}\cong\tilde{\tau}^\sigma\otimes\mu^\sigma\cong(\tilde{\tau}\otimes\mu)^\sigma \]
	and so $\tilde{\tau}\otimes\mu$ is either $\GL_2(K)$-distinguished or $(\GL_2(K),\omega_{EK/K})$-distinguished, but not both. We  map $\chi$ to $\chi_K$ or $\chi_K\omega_{EK/K}$ accordingly.
	Clearly the above two maps are inverses of each other. This finishes the proof.
\end{proof}
\begin{lem}
	If $p\neq2$, then $Z_{\tilde{\tau}}$ is a singleton.
\end{lem}
\begin{proof}
	Since $\chi\in Y_{\tilde{\tau}}$, $\chi$ is a quadratic character and $\chi=\chi^s$ where $\Gal(EK/E)\cong\Gal(K/F)=\langle s\rangle$. If $\chi\in Z_{\tilde{\tau}}$, then $\chi^\sigma=\chi=\chi^s$. So $\chi=\chi_F\circ N_{EK/F}$, where $\chi_F$ is a quadratic character of $F^\times$. Thus $\chi=\mathbf{1}$ if $p\neq2$.
\end{proof}
If $p=2$, then $Z_{\tilde{\tau}}$ may not be a singleton. Suppose that $|Z_{\tilde{\tau}}|>1$. Set $$\GL_2(EK)^+=\GL_2(K)\cdot\GL_2(EK)^\natural.$$ Moreover, $\tilde{\tau}|_{\GL_2(EK)^+}$ is reducible. Thanks to \cite[Proposition 3.1]{anandavardhanan2003distinguished}, if an irreducible representation
$\Sigma$ of $\GL_2(EK)^+$ is $\GL_2(K)^\natural$-distinguished, then $\Sigma$ has a Whittaker model with respect to a nontrivial character $\psi_0$ of $EK/K$. Furthermore, exactly one constituent of the restriction of $\tilde{\tau}$ to $\GL_2(EK)^+$  is $\psi_0$-generic.
\begin{coro}
	If $p\neq2$ and $\pi$ is $\SO(V)$-distinguished, then $\dim\Hom_{\SO(V)}(\pi,\mathbb{C} )=1 $.
\end{coro}

If $p=2$, then \[\dim\Hom_{\SO(V)}(\pi,\mathbb{C})=\frac{|X_{\tilde{\tau}}|}{|Y_{\tilde{\tau}}|/|Z_{\tilde{\tau}}|} \]
due to \cite[Lemma 22]{prasad2015arelative}.
%since the map $\chi\mapsto \chi|_{K^\times}$ from $Y_{\tilde{\tau}}$ from $X_{\tilde{\tau}}$ is surjective with kernel $Z_{\tilde{\tau}}$.

\begin{prop}
	Assume that $\tau$ is a principal series representation of $\GL_2(EK)^\natural$ with associated representation $\pi$ of $\SO(V\otimes_F E)$, distinguished by $\GSpin(V)$. Then $\dim\Hom_{\SO(V)}(\pi,\mathbb{C})\in\{1,2,3,4 \}$.
\end{prop}
\begin{proof} If $p\neq2$, then $\omega_{\tilde{\tau}}=\mathbf{1}$. There are two cases:
	\begin{itemize}
		\item 
			If $\tilde{\tau}=\pi(\mu^{-1},\mu^\sigma)$, then $\tilde{\tau}|_{\GL_2(EK)^\natural}$ is irreducible unless
		$\mu^\sigma\mu=\chi_F\circ N_{EK/F}$ with $\chi_F^2=\omega_{E/F}$. So $$\dim\Hom_{\SO(V) }(\pi,\mathbb{C})=\begin{cases}
		1,&\mbox{if }\mu^\sigma\mu=\chi_F\circ N_{EK/F},\chi_F^2=\omega_{E/F};\\
		2,&\mbox{otherwise.}
		\end{cases}
		%2=|X_{\tilde{\tau}}|.
		$$
		\item  If  $\tilde{\tau}=\pi(\mu_1,\mu_2)$ with $\mu_1\neq \mu_2$ and $\mu_1|_{K^\times}=\mu_2|_{K^\times}=\mathbf{1}$, then $\tilde{\tau}|_{\GL_2(EK)^\natural}$ is irreducible and $$\dim\Hom_{\SO(V) }(\pi,\mathbb{C})=\dim\Hom_{\GSpin(V)}(\tau,\mathbb{C})=1.$$
	\end{itemize}	
	If $p=2$, then $\tilde{\tau}|_{\GL_2(EK)^+}$ may be reducible. There is only one constituent in 
	$\tilde{\tau}|_{\GL_2(EK)^+}$ distinguished by $\GSpin(V)$. Moreover,
	if $\mu^\sigma\mu=\chi_F\circ N_{EK/F}\neq\mathbf{1}$ with $\chi_F^2=\mathbf{1}$, there are two representations in the $L$-packet $\Pi_{\phi_{\pi}}$ and only one of them is $\SO(V)$-distinguished with multiplicity $4$. Similarly, if $\mu_1=\mu_2\cdot\chi_F\circ N_{EK/F}\neq\mu_2$ with $\mu_1|_{K^\times}=\mathbf{1}=\chi_F^2$, then $\dim\Hom_{\SO(V)}(\pi,\mathbb{C}) =3$.
%	$Y_{\tilde{\tau}}=Z_{\tilde{\tau}}$ and $\dim\Hom_{\GL_2(K)^\natural }(\tau,\mathbb{C})=|X_{\tilde{\tau}}|$.
\end{proof}
Note that there exists a unique pure inner form of $\SO(V)$.
\begin{thm}[The Prasad conjecture for $\SO(V)$] Let $\pi$ be an irreducible representation of $\SO(V\otimes_F E)=\GSpin_{3,1}(E)/E^\times$.
	There exists an identity
	\[2\cdot\dim\Hom_{\SO(V)}(\pi,\mathbb{C})=\sum_{\phi\in F(\phi_\pi)}m(\lambda,\phi)\frac{\deg\Phi(\phi)}{d_0(\phi)} \]
	where $F(\phi_{\pi})=\{\phi:WD_F\longrightarrow \Oo_4(\mathbb{C}):\det\phi=\omega_{K/F}\mbox{ and }\phi|_{WD_E}=\phi_{\pi} \}$.
\end{thm}
\begin{proof}
	Suppose that $p\neq2$ and $\pi$ is $\SO(V)$-distinguished and square-integrable.  Then $$\dim\Hom_{\SO(V)}(\pi,\mathbb{C})=|X_{\tilde{\tau}}|.$$ Thanks to \cite[Theorem 2.5.2]{lu2016new},
	there exists a representation $\rho$ of $\GL_2(K)$ such that $$\phi_{\rho}|_{WD_{E}}=\phi_{\tilde{\tau}}:WD_E\longrightarrow(\GL_2(\mathbb{C})\times\GL_2(\mathbb{C}))\rtimes\Gal(EK/E)\cong{}^L(R_{K/F}\GL_2).$$
	Note that there is a natural group embedding $\GSpin(V)\hookrightarrow R_{K/F}(\GL_2)$. Let $$p:{}^L(R_{K/F}\GL_2)\longrightarrow \GO_4(\mathbb{C})$$ be the natural quotient map between two $L$-groups. In fact,   the Asai lift $As_{K/F}(\rho)$ (in the sense of \cite[\S7]{gan2011symplectic}) of $\rho$ lies in the $L$-packet $\Pi_{p\circ\phi_{\rho}}$ of $\GSpin(V)\subset\GL_4(F)$. 
	Then $$\phi=As_{K/F}(\rho)=p\circ \phi_{\rho}:WD_F\rightarrow \Oo_4(\mathbb{C})$$
	 and $\phi|_{WD_E}=\phi_{\pi}$. There are several subcases:
	\begin{itemize}
		\item If $\tilde{\tau}$ is square-integrable, then there are two lifts $\phi$ and $\omega_{E/F}\otimes\phi$ for $\phi_{\pi}$. 
		\item If $\phi_{\tilde{\tau}}=\mu^{-1}+\mu^{\sigma}$, then $\mu=\mu^\sigma$ since $\omega_{\tilde{\tau}}=\mathbf{1}$. There exists a character $\chi_K$ of $W_K$ such that
		$\chi_K|_{W_{EK}}=\mu$. Set $$\phi_{\rho}=\chi_K+\chi_K^{-1}.$$ Then $p\circ\phi_{\rho}=As_{K/F}(\rho)\in F(\phi_\pi ) $. If $\mu^2\neq \chi_F|_{W_{EK}}$ with $\chi_F^2=\omega_{E/F}$, then
		$\phi_{\rho'}=\chi_K\omega_{E/F}|_{W_K}\oplus \chi_K^{-1}$ induces another lifts $p\circ\phi_{\rho'}\in F(\phi_\pi ) $. Therefore
		\[F(\phi_{\pi})=\begin{cases}	\{\phi,\omega_{E/F}\otimes \phi \},&\mbox{ if }\mu^2=\chi_F|_{W_{EK}},\chi_F^2=\omega_{E/F};\\
		\{\phi,\omega_{E/F}\otimes  \phi,As_{K/F}(\rho'),\omega_{E/F}\otimes  As_{K/F}(\rho')  \},&
	\mbox{ otherwise. }
		\end{cases} \]
		Here $\phi=p\circ\phi_{\rho}$.
		Suppose that $\mu^2=\chi_F|_{W_{EK}}$ with $\chi_F^2=\omega_{E/F}$. Then $S_{\phi_{\pi}}=\mu_2$ and \[S_{\phi }=1\hookrightarrow S_{\phi_{\pi}}. \]
		Inside the $L$-packet $\Pi_{\phi_{\pi}}$, there are two elements and both are $\SO(V)$-distinguished.
		%There are four choices for $\phi_\rho$ unless $\chi_K^2=\omega_{E/F}|_{W_K}$ or $\mathbf{1}$, which are
%		\[\chi_K+\chi_K^{-1},\chi_K\omega_{E/F}|_{W_K}+\chi_K^{-1},\chi_K+\chi_K^{-1}\omega_{E/F}|_{W_K}\mbox{  and  }\chi_K\omega_{E/F}|_{W_K}+\chi_K^{-1}\omega_{E/F}|_{W_K}. \] 
%		There are two parameters of $\SO(V)$ for $p\circ\phi_{\rho} $, say $\phi_1$ and $\phi_2$ satisfying $\phi_i|_{WD_E}=\phi_\pi$, including the cases $\chi_K^2=\omega_{E/F}|_{W_K}$ and $\chi_K^2=\mathbf{1}$. Then the parameters $\phi_1,\phi_2,\phi_1\otimes\omega_{E/F}$ and $\phi_2\otimes\omega_{E/F}$ are what we need.
 		\item If $\phi_{\tilde{\tau}}=\mu_1+\mu_2$ with $\mu_1\neq\mu_2$ and $\mu_1|_{K^\times}=\mu_2|_{K^\times}=\mathbf{1}$, then there exists a square-integrable representation $\rho$ of $\GL_2(K)$ such that $\phi_{\rho}|_{WD_E}=\phi_{\tilde{\tau}}$. Hence
		$F(\phi_{\pi})=\{\phi,\omega_{E/F}\otimes \phi \}$.
	\end{itemize}
If $p=2$, then $\omega_{\tilde{\tau}}=\chi_F'\circ N_{EK/F}$ with $(\chi_F')^2=\mathbf{1}$. 
%If $\omega_{\tilde{\tau}}=\mathbf{1}$, then the proof is the same as the case $p\neq2$. 
Suppose $\omega_{\tilde{\tau}}\neq\mathbf{1}$. Then $\chi_F'$ corresponds to a quadratic field $E_4$ of $F$ which is not contained in $EK$, which happens only when $p=2$. 
%Moreover, there exists a  character $\chi$ of $W_{F}$ such that ${\mu}|_{E^\times}=\chi^{-1}\circ N_{E/F} $ and 
\begin{itemize}
	\item If $\tilde{\tau}$ is square-integrable, then $\dim\Hom_{\SO(V) }(\pi,\mathbb{C})$ is either $1$, $2$ or $4$. This case is very similar to the situation when we study the Prasad conjecture for $\SL_2$ in \cite{L2018pacific}. So we omit it here.
	\item If $\phi_{\tilde{\tau}}=\mu^{-1}+\mu^\sigma$, then $$\phi_\pi|_{WD_{EK}}=\frac{1}{\mu\mu^s}\begin{pmatrix}
	\mathbf{1}\\&\mu\mu^\sigma
	\end{pmatrix}\otimes\begin{pmatrix}
	\mathbf{1}\\&\mu^s \mu^{s\sigma}
	\end{pmatrix}.$$ Note that
	$\mu^\sigma=\mu\cdot{\chi_F'}|_{W_{EK}}$ and so $(\mu\mu^s)^\sigma=\mu\mu^s$. Assume that $\mu\mu^s={\chi_F''}|_{W_{EK}}$ for a character $\chi_F''$ of $W_F$.
	Suppose that $\dim\Hom_{\SO(V)}(\pi,\mathbb{C})=2 $. Then \[F(\phi_{\pi})=\{(\chi_F'')^{-1}\otimes p\circ\phi_{\rho_1},(\chi_F'')^{-1}\omega_{E/F}\otimes p\circ \phi_{\rho_1},(\chi_F'')^{-1}\otimes p\circ\phi_{\rho_2},(\chi_F'')^{-1}\omega_{E/F}\otimes p\circ\phi_{\rho_2} \} \]
	where $\phi_{\rho_1}=\omega_{E/F}|_{W_K}+\mu|_{K^\times}$ and $\phi_{\rho_2}=\mathbf{1}+\mu|_{K^\times}$.
	
	If $\dim\Hom_{\SO(V)}(\pi,\mathbb{C})=4 $, then $\mu\mu^\sigma=\chi_F\circ N_{EK/F}$ with $\chi_F^2=\mathbf{1}$. On the Galois side,
	\[F(\phi_{\pi})=\bigcup_{i=1}^2\bigcup_{z_1,z_2\in\{0,1\}}\{(\chi_F'')^{-1}\chi_F^{z_1}\omega_{E/F}^{z_2}\otimes As_{K/F}(\rho_i) \} .\]
 There are two elements inside the $L$-packet $\Pi_{\phi_{\pi}}$ and only one member is $\SO(V)$-distinguished. For any $\phi\in F(\phi_{\pi})$, $$S_{\phi}=\mu_2\cong S_{\phi_{\pi}}$$
 and there is only one character $\lambda$ of $S_{\phi_{\pi}}$ such that  $\lambda|_{S_{\phi}}$ contains the trivial character.
	\item If $\phi_{\tilde{\tau}}=\mu_1+\mu_2 $ with $\mu_1=\mu_2\cdot \chi_F\circ N_{EK/F}$ and $\chi_F^2=\mathbf{1}$, then $\mu_1^2=(\chi_F\chi_F')|_{W_{EK}}$  since $\mu_1\cdot\mu_2=\chi_F'|_{WD_{EK}}$. So
	 $$\phi_{\pi}|_{WD_{EK}}=\mu_1\mu_1^s\begin{pmatrix}
	\mathbf{1}\\&\chi_F|_{W_{EK}}
	\end{pmatrix}\otimes\begin{pmatrix}
	\mathbf{1}\\&\chi_F|_{W_{EK}}
	\end{pmatrix}.$$
	Note that $\mu_1=\nu_1^\sigma\nu_1^{-1}$ for a character $\nu_1$ of $(EK)^\times$ and so $$(\mu_1\mu_1^s)^\sigma=\mu_1^{-1}(\mu_1^s)^{-1}=\nu_1\nu_1^s(\nu_1^\sigma\nu_1^{s\sigma})^{-1}=\nu_1^\sigma\nu_1^{s\sigma}(\nu_1\nu_1^s )^{-1}=\mu_1\mu_1^s.$$
	There exists $\chi_F'''$ such that $\mu_1\mu_1^s=\chi_F'''\circ N_{EK/F}$
	 and
there are $6$
 lifts for $\phi_{\pi}$. Set $\rho_1=\mathbf{1}+ \chi_F|_{W_K}$ and $\rho_2=\omega_{E/F}|_{W_K}+\chi_F|_{W_K}$	and $\rho_3=Ind_{W_{EK}}^{W_K}\nu$ where $\nu^\sigma=\nu\cdot\chi_F\circ N_{EK/F}$, then $$F(\phi_{\pi})=\bigcup_{i=1}^3 \{\chi_F''\otimes As_{K/F}(\rho_i),\chi_F''\omega_{E/F}\otimes As_{K/F}(\rho_i) \}.$$ 
 In fact, $\nu=\nu_1^\sigma\nu_2$ where $\mu_i=\nu_i^\sigma\nu_i^{-1}$ for $i=1,2$ and $\dim\Hom_{\SO(V)}(\pi,\mathbb{C})=3 $ in this case.
\end{itemize}
We have finished the proof.
\end{proof}
If $K=E$, then the representation $\tilde{\tau}=\pi_1\times\pi_2$ of $\GL_2(E)\times\GL_2(E)$ is $\GSpin(V)$-distinguished  if $\pi_1^\vee=\pi_2^\sigma$, where $\pi_1$ and $\pi_2$ are irreducible representations of $\GL_2(E)$. Set 
\[X_{\tilde{\tau}}=\{\chi:E^\times/F^\times\longrightarrow\mathbb{C}^\times|\Hom_{\GL_2(E)}(\tilde{\tau},\chi\circ\det)\neq0
 \},\]
 $Y_{\tilde{\tau}}=\{\chi:E^\times\rightarrow\mathbb{C}^\times|\pi_i\otimes\chi=\pi_i\mbox{ for }i=1,2 \}$ and $Z_{\tilde{\tau}}=\{\chi\in Y_{\tilde{\tau}}:\chi|_{E^1}=\mathbf{1} \}$. Note that
 %There exists a bijection between $X_{\tilde{\tau}}$ and $Z_{\tilde{\tau}}$ if $\tilde{\tau}$ is square-integrable.
 \[\xymatrix{\GSpin_{2,2}(E)\ar[r]&\GL_2(E)\times\GL_2(E)\\ \GSpin(V)\ar[u]\ar[r]&(R_{E/F}\GL_2)(F)\ar[u]_{g\mapsto(g,\sigma(g))} } \]
 where $\GSpin_{2,2}(E)=\{(g_1,g_2)\in\GL_2(E)\times\GL_2(E):\det g_1=\det g_2 \}$.
 
  Assume that $\pi$ is an irreducible representation of $\SO(V\otimes_F E)$ distinguished by $\SO(V)$, associated to $\tilde{\tau}=\pi_1\times\pi_2$ with $\omega_{\pi_1}=\omega_{\pi_2}^{-1}$ and $\pi_1^\vee=\pi_2^\sigma$, where $\omega_{\pi_i}$ is the central character of $\pi_i$. Then
 \[\dim\Hom_{\SO(V)}(\pi,\mathbb{C})=\frac{|X_{\tilde{\tau}}|}{|Y_{\tilde{\tau}}|/|Z_{\tilde{\tau}}|}.  \]
 Suppose $\chi\in X_{\tilde{\tau}}$. Then $\pi_1\otimes\chi\cong\pi_1$ and $\chi|_{F^\times}=\mathbf{1}$. Let $\tau\subset\tilde{\tau}|_{\GSpin_{2,2}(E)} $ be a $\GSpin(V)$-distinguished representation of $\GSpin_{2,2}(E)$. Denote $\tau^\epsilon$ the conjugate representation of $\tau$, i.e.,
 \[\tau^\epsilon(g_1,g_2)=\tau(g_1,\begin{pmatrix}
 \epsilon&\\&1
 \end{pmatrix}^{-1}g_2\begin{pmatrix}
 \epsilon\\&1
 \end{pmatrix} ) \]
 for $(g_1,g_2)\in\GSpin_{2,2}(E)$ and $\epsilon\in F^\times \setminus N_{E/F}E^\times$.
 \begin{lem}
 	Suppose that $\tau$ is an irreducible $\SO(V)$-distinguished representation of $\GSpin_{2,2}(E)$ contained in $\tilde{\tau}|_{\GSpin_{2,2}(E)}$ and $|Y_{\tilde{\tau}}|=|Z_{\tilde{\tau}}|\neq|X_{\tilde{\tau}}|$. Then 
 	\[\dim\Hom_{\SO(V')}(\tau,\mathbb{C})=0  \]
 	where $H'(F)=\SO(V')$ is the nontrivial pure inner form of $H(F)=\SO(V)$.
 \end{lem}
\begin{proof}
	By the assumption, there exists $\chi\in Z_{\tilde{\tau}}$ such that $\chi=\chi_F\circ N_{E/F}$ with $\chi_F^2=\omega_{E/F}$.
	Suppose that $$G^\natural=\{(g_1,g_2)\in\GL_2(E)\times\GL_2(E):N_{E/F}(\det g_1)=N_{E/F}(\det g_2) \}.$$ There is only one constituent in $\tilde{\tau}|_{G^\natural}$ distinguished by $\GSpin(V)$. Since $\tau$ is $\SO(V)$-distinguished, $\tau^\epsilon$ is not $\SO(V)$-distinguished. Thus,
	\[\dim\Hom_{\SO(V')}(\tau,\mathbb{C})=\dim\Hom_{\SO(V)}(\tau^\epsilon,\mathbb{C})=0.   \]
\end{proof}
  \begin{thm}
  	The Prasad conjecture holds for $\SO(V)$ when the discriminant aglebra of $V$ equals to $E$.
  \end{thm}
\begin{proof}
	Assume that $\pi$ is square-integrable. Let $\tilde{\tau}=\pi_1\times\pi_2$ with $\omega_{\pi_1}=\omega_{\pi_2}^{-1}$ and $\pi_1^\vee=\pi_2^\sigma$. Then each central character $\omega_{\pi_i}$ is $\Gal(E/F)$-invariant and there exists a character $\chi_F$ such that $\chi_F\circ N_{E/F}= \omega_{\pi_1}^{-1}$. Then
	\begin{itemize}
		\item If $\pi_1\cong\pi_2$, then $\pi_1^\vee\cong\pi_1^\sigma$ and $\phi_{\pi}=\phi_{\pi_1}\otimes\phi_{\pi_1}=(\det\phi_{\pi_1}\otimes Ad(\phi_{\pi_1}))\oplus\det\phi_{\pi_1}$. Moreover $\det\phi_{\pi_1}$ is a quadratic character of $W_E$ and is conjugate-orthogonal. So $\chi_F^2=\mathbf{1}$. Thanks to \cite[Theorem 1.2]{L2018pacific}, there exists an identity
		\[\dim\Hom_{\SO(V)}(\pi,\mathbb{C})=|\{\phi:WD_F\longrightarrow\PGL_2(\mathbb{C})|\phi|_{WD_E}=Ad(\phi_{\pi_1}) \}|=|\{\phi_1,\phi_2,\cdots,\phi_r \}|  \]
		both equal to $r=\frac{|X_{\tilde{\tau}}|}{|Y_{\tilde{\tau}}|/|Z_{\tilde{\tau}}|}$ (which may be either $1$, $2$ or $4$). Then
		$$F(\phi_{\pi})=\bigcup_{i=1}^r\{\chi_F\otimes(\phi_i\oplus\omega_{E/F}),\omega_{E/F}\chi_F\otimes(\phi_i\oplus\omega_{E/F}) \}.$$
		\item If $\pi_2=\pi_1\otimes\chi$ for a character $\chi$ of $E^\times$ and $\pi_1\neq\pi_2$, then 
		$X_{\tilde{\tau}}\subset Z_{\tilde{\tau}}\subset Y_{\tilde{\tau}}$,
		$$\phi_{\pi}=(\mathbb{C}\oplus Ad(\phi_{\pi_1}))\otimes\chi\det\phi_{\pi_1},$$ and $\pi_1\otimes\chi=(\pi_1^\sigma)^\vee$. So $\chi\cdot
		\omega_{\pi_1}$ is a quadratic character of $E^\times$. If $\chi=\chi^\sigma$, then it coincides with the previous case. Assume that $\chi\neq\chi^\sigma$. Then $\chi^\sigma\chi^{-1}$ is a nontrivial quadratic character and $\chi^\sigma\chi^{-1}\in Z_{\tilde{\tau}}\cap X_{\tilde{\tau}} $ since $\pi_1^\vee=\pi_2^\sigma=\pi_1^\sigma\otimes\chi^\sigma=\pi_2^\vee\otimes\chi^\sigma=\pi_1^\vee\otimes\chi^\sigma\chi^{-1}$. Therefore
		\[\dim\Hom_{\SO(V)}(\pi,\mathbb{C})=\begin{cases}
		1,&\mbox{ if }|Y_{\tilde{\tau}}|=4,|Z_{\tilde{\tau}}|=2;\\
		|X_{\tilde{\tau}}|,&\mbox{ if }|Y_{\tilde{\tau}}|=|Z_{\tilde{\tau}}|.
		\end{cases}  \]
		On the Galois side, if $\dim\Hom_{\SO(V)}(\pi,\mathbb{C})=1 $, then $|F(\phi_{\pi})|=|Z_{\tilde{\tau}}|$ and
		\[F(\phi_{\pi})=\{\chi_F As_{E/F}(\phi_1)\nu,\nu\circ N_{E/F}\in Z_{\tilde{\tau}} \} \]
		where $As_{E/F}(\phi_1)=\omega_{E/F}As_{E/F}(\phi_1)$ due to the fact $tr(\phi_1(\ell^2))=0$, where $\ell\in W_F\setminus W_E$ is fixed.
		
		If $|Y_{\tilde{\tau}}|=|Z_{\tilde{\tau}}|\neq|X_{\tilde{\tau}}|$, then $\dim\Hom_{\SO(V')}(\pi,\mathbb{C})=0 $ and $$F(\phi_{\pi})=\bigcup_{z\in\{0,1\}} \{\chi_F As_{E/F}(\phi_1)\nu_i^{z},\nu_i\circ N_{E/F}\in Z_{\tilde{\tau}} \}.$$
	%	If $|Z_{\tilde{\tau}}|=4=2|X_{\tilde{\tau}}|$, then $\Hom_{\SO(V')}(\pi,\mathbb{C})=0$, $As_{E/F}(\phi_1)=\omega_{E/F}As_{E/F}(\phi_1)$ and
%	\[F(\phi_{\pi})=\{\chi_F As_{E/F}(\phi_1)\nu,\nu\circ N_{E/F}\in X_{\tilde{\tau}} \}. \]
		%Note that $\chi\cdot\omega_{\pi_2}|_{F^\times}$ is trivial, then $\chi\cdot\omega_{\pi_2}$ is $\Gal(E/F)$-invariant and so $\chi$ is $\Gal(E/F)$-invariant. 
		%Assume that $\chi=\mu\mu^\sigma$ for a character $\mu$ of $E^\times$, then $$(\pi_2\otimes\mu)^\vee=\pi_2^\vee\otimes\mu^{-1}=\pi_2^\sigma\otimes\chi\mu^{-1}=(\pi_2\otimes\mu)^\sigma.$$
		%Thus $\pi_2\otimes\mu$ is $\SL_2(F)$-distinguished and $Ad(\phi_{\pi_2})=Ad(\phi_{\pi_2\otimes\mu})$. Therefore $$|F(\phi_{\pi})|=2\dim\Hom_{\SO(V)}(\pi,\mathbb{C}). $$
		\item If $\pi_1\neq\pi_2\otimes\chi$ for any character $\chi$, then $|Y_{\tilde{\tau}}|$ is either 1 or 2  and $\phi_{\pi}=(\phi_{\pi_1}\otimes\phi_{\pi_1^\sigma})\frac{1}{\det\phi_{\pi_1}}$. Set 
		\[\phi(g)=\begin{cases}
		\phi_\pi(g)&\mbox{if }g\in WD_E;\\1\otimes\phi_{\pi_1}(\ell^2)\chi_F(\ell)\cdot \sigma&\mbox{if }g=\ell\in W_F\backslash W_E.
		\end{cases} \]
		Then $$\phi=\chi_F As_{E/F}(\pi_1)\in F(\phi_{\pi}).$$ If $Z_{\tilde{\tau}}=\langle\chi\circ N_{E/F}:\chi^2=\omega_{E/F} \rangle$, then $\dim\Hom_{\SO(V)}(\pi,\mathbb{C})=\frac{1}{2/2} =1$, $\Hom_{\SO(V')}(\pi,\mathbb{C})=0$ and $$F(\phi_{\pi})=\{As_{E/F}(\pi_1)\otimes\chi_F \}.$$
		Because $\pi_1\otimes\chi\circ N_{E/F}\cong\pi_1$ with $\chi^2=\omega_{E/F}$, 
		\[tr( \phi_{\pi_1}(\ell^2))=\chi(s^2)tr(\phi_{\pi_1}(\ell^2))=-tr(\phi_{\pi_1}(\ell^2)). \]
		Then $tr(\phi_{\pi_1}(\ell^2) )=0$ and $As_{E/F}(\pi_1)=\omega_{E/F}\cdot As_{E/F}(\pi_1)$.
		\par
		 Suppose that $Z_{\tilde{\tau}}=\langle\chi\circ N_{E/F}:\chi^2=\mathbf{1}\rangle$; then $\dim\Hom_{\SO(V)}(\pi,\mathbb{C})=2 $ and
		$$F(\phi_{\pi})=\bigcup_{z,z'\in\{0,1\}} \{\phi\otimes\chi^z,~\phi\otimes\chi^{z'}\omega_{E/F} \}.$$
		If $|Y_{\tilde{\tau}}|=|Z_{\tilde{\tau}}|=2$, then only one member inside the $L$-packet $\Pi_{\phi_{\pi}}$ is $\SO(V)$-distinguished.
	\end{itemize}
If $\pi$ is not square-integrable, then both $\pi_1$ and $\pi_2$ are principal series representations. Suppose that $$\phi_{\pi_i}=\mu_{1i}\oplus\mu_{2i}$$ for $i=1,2$. Then $\mu_{11}\mu_{21}\mu_{12}\mu_{22}=\mathbf{1}$ and $\mu_{j1}\mu_{j2}^\sigma=\mathbf{1}$ for $j=1,2$. Note that
$Y_{\tilde{\tau}}=Z_{\tilde{\tau}}$ and $$\dim\Hom_{\SO(V)}(\pi,\mathbb{C}) =\begin{cases}
2,&\mbox{ if }\mu_{11}\mu_{21}^{-1}=\nu\circ N_{E/F},\nu^2=\mathbf{1};\\
1,&\mbox{ otherwise.}
\end{cases}$$ The Langlands parameter
\[\phi_{\pi}=\mu_{11}\mu_{12}\oplus\mu_{11}\mu_{22}\oplus\mu_{21}\mu_{12}\oplus\mu_{21}\mu_{22}  \]
and $\mu_{11}\mu_{21}$ is $\Gal(E/F)$-invariant. There exists a character $\chi_F$ such that $\mu_{11}\mu_{21}=\chi_F^{-1}\circ N_{E/F}$. Then
\[F(\phi_{\pi} )=\{\chi_F\omega_{E/F}^zAs_{E/F}(\pi_1)\nu_i^{z'},\mbox{ where }z,z'\in\{0,1\}\mbox{ and }\nu_i\circ N_{E/F}\in X_{\tilde{\tau}} \}. \]
Note that if $\mu_{11}\mu_{21}^{-1}=\nu\circ N_{E/F}$ with $\nu^2=\omega_{E/F}$. Then $\Hom_{\SO(V')}(\pi,\mathbb{C})=0$, $$As_{E/F}(\pi_1)=\omega_{E/F}\otimes As_{E/F}(\pi_1)$$ and $|F(\phi_{\pi})|=1$.
Thus the Prasad conjecture for $\SO(V)$ holds.
\end{proof}

\subsection{The Prasad conjecture for $\GSO_{2,2}$} Let $V$ be a $4$-dimensional quadratic space over $F$ with trivial discriminant and split (resp. non-split) special orthogonal group $\SO(V)$, denoted by $\SO_{2,2}(F)$ (resp. $\SO_{4,0}(F)$). The special orthogonal similitude group (denoted by $\GSO_{2,2}(F)$) is given by $$\GSO_{2,2}(F)=\Bigg\{g\in\GL_4(F)| g\begin{pmatrix}
&&&1\\&&1\\&1\\1
\end{pmatrix} g^t=\lambda_V(g)\begin{pmatrix}
&&&1\\&&1\\&1\\1
\end{pmatrix}   \Bigg\}\cong \frac{\GL_2(F)\times\GL_2(F) }{\{(t,t^{-1}):t\in F^\times\}}. $$
There is a unique nontrivial pure inner form of $\GSO_{2,2}$ in $H^1(F,\GSO_{2,2})$, denoted by $\GSO_{4,0}$.

 If $H=\GSO_{2,2}$, then $\chi_H=\omega_{E/F}$ and $$H^{op}(F)=\frac{\Uu_2\times\Uu_{2} }{\{(e,e^{-1}):e\in \Uu_1\}}.$$
We denote by $\mathrm{USO}_{2,2}$ the group $H^{op}$ if $H=\GSO_{2,2}$. Then ${}^L\mathrm{USO}_{2,2}=\GSpin_{2,2}(\mathbb{C})\rtimes\langle \sigma\rangle$ where \[\GSpin_{2,2}(\mathbb{C})=\{(g_1,g_2)\in\GL_2(\mathbb{C})\times\GL_2(\mathbb{C}):\det(g_1)=\det(g_2)  \} \] and the action of $\sigma$ on $\GSpin_{2,2}(\mathbb{C})$ is given by
\[\sigma(g_1,g_2)=(\det(g_1)^{-1}\cdot g_1,\det(g_2)^{-1}\cdot g_2) \]
  for  $(g_1,g_2)\in\GSpin_{2,2}(\mathbb{C}).$ 
\begin{thm}\label{GSO(2,2)}
	Let $E$ be a quadratic field extension of a nonarchimedean local field $F$ with associated Galois group $\Gal(E/F)=\{1,\sigma \}$ and associated quadratic character $\omega_{E/F}$. Assume that $\Sigma=\pi_1\boxtimes\pi_2$ with $\omega_{\pi_1}=\omega_{\pi_2}$ is an irreducible generic representation of $\GSO_{2,2}(E)$ with an $L$-parameter $\phi_\Sigma$. Then
	\begin{enumerate}[(i)]
		\item If $\Sigma$ is $(H(F),\chi_H)$-distinguished, then $\Sigma^\vee\cong\Sigma^\sigma$ and there exists a Langlands parameter $\phi$ of $\mathrm{USO}_{2,2}(F)$ such that $\phi|_{WD_E}=\phi_\Sigma$.
		\item If $\phi_\Sigma=\phi_{\tau}|_{WD_E}$ for an irreducible representation $\tau$ of $\mathrm{USO}_{2,2}(F)$, then $\Sigma$ is $(H(F),\chi_H)$-distinguished.
		\item Suppose $F(\phi_\Sigma)=\{\phi:WD_F\longrightarrow{}^LH^{op}:\phi|_{WD_E}=\phi_\Sigma \}$. Then there is an identity
		\begin{equation}\label{prasad:GSO}
		\dim\Hom_{\GSO_{2,2}(F)}(\Sigma,\omega_{E/F})+\dim\Hom_{\GSO_{4,0}(F)}(\Sigma,\omega_{E/F} )=\sum_{\phi\in F(\phi_\Sigma)} \frac{\deg\Phi(\phi)}{d_0(\phi)}  . \end{equation}
	\end{enumerate}
\end{thm}
\begin{proof} Recall that $\phi_\Sigma=(\phi_{\pi_1},\phi_{\pi_2})$ with $\det(\phi_{\pi_1})=\det(\phi_{\pi_2})$.
	\begin{enumerate}[(i)]
		\item Note that $$\dim\Hom_{H(F)}(\Sigma,\chi_H)=\dim\Hom_{\GL_2(F)}(\pi_1,\omega_{E/F})\cdot\dim\Hom_{\GL_2(F)}(\pi_2,\omega_{E/F}) .$$ If $\Sigma$ is $(H(F),\chi_H)$-distinguished, then both $\phi_{\pi_1}$ and $\phi_{\pi_2}$ are conjugate-symplectic and so $\pi_i^\vee=\pi_i^\sigma$. Thus $\Sigma^\vee\cong\Sigma^\sigma$ and there exist two representations $\rho_1$ and $\rho_2$ of $\Uu_2$ such that $\phi_{\rho_i}|_{WD_E}=\phi_{\pi_i}$ for $i=1,2$. Then $\phi=(\rho_1,\rho_2)$ is what we want.
		\item If $\phi_\Sigma=\phi_{\tau}|_{WD_E}=(\rho_1,\rho_2)|_{WD_E}$, then $\phi_{\pi_i}=\rho_i|_{WD_E}$ is conjugate-symplectic for $i=1,2$. Therefore each $\pi_i$ is $(\GL_2(F),\omega_{E/F})$-distinguished and $\Sigma$ is $(H(F),\chi_H)$-distinguished. 
		\item Suppose that $\Sigma$ is $(H(F),\chi_H)$-distiguished. Otherwise, both sides of \eqref{prasad:GSO} will be zero. We divide them into two cases:
		\begin{itemize}
			\item 
		 If neither $\phi_{\pi_1}$ nor $\phi_{\pi_2}$ is of the form $\chi_1+\chi_2$ with $\chi_1|_{F^\times}=\chi_2|_{F^\times}=\omega_{E/F}$ and $\chi_1\neq\chi_2$, then there are two parameters $(\phi_{\rho_1},\phi_{\rho_2})$ and $(\phi_{\rho_1}\omega_{E/F},\phi_{\rho_2})$ of $H^{op}(F)$ in $F(\phi_{\Sigma})$.
		 \item
		  If $\phi_{\pi_1}=\chi_1+\chi_2$ with $\chi_1|_{F^\times}=\chi_2|_{F^\times}=\omega_{E/F}$ and $\chi_1\neq\chi_2$, then there is only one lift for $\phi_{\Sigma}$. In this case, $\dim\Hom_{H(F)}(\Sigma,\chi_H)=1$ and $\Hom_{D}(\pi_1,\omega_{E/F} )=0 $ where $D$ is the inner form of $\GL_2(F)$. Thus $\dim\Hom_{\GSO_{4,0}(F)}(\Sigma,\omega_{E/F})=0$.
		\end{itemize}
		   In all cases, $\deg\Phi(\phi)=1$ for $\phi\in F(\phi_{\Sigma})$.
	\end{enumerate}
Then we have finished the proof.
\end{proof}
\subsection{The Prasad conjecture for $\SO_{2,2}$}
 Recall that if $H=\SO_{2,2}$, then $\chi_H=\mathbf{1}$ and $H^{op}=\SO_{2,2}$.
\begin{proof}
	[Proof of Theorem \ref{localmain} when $H=\SO_{2,2}$]
\begin{enumerate}[(i)]
	\item If $\pi$ is $\SO_{2,2}(F)$-distinguished, then there exists a representation $\Sigma=\pi_1\boxtimes\pi_2$ of $\GSO_{2,2}(E)$, distinguished by $\GSO_{2,2}(F)$, such that $\pi\subset \Sigma|_{\SO_{2,2}(E)}$. So $\phi_{\pi}=\phi_{\pi_1}^\vee\otimes\phi_{\pi_2}$ and both $\phi_{\pi_1}$ and $\phi_{\pi_2}$ are conjugate-orthogonal. Therefore
	\[\phi_{\pi^\sigma}=\phi_{\pi_1^\sigma}^\vee\otimes\phi_{\pi_2^\sigma}=\phi_{\pi_1}\otimes\phi_{\pi_2}^\vee=\phi_{\pi}^\vee. \]
	Suppose that $\omega_{\pi_1}=\omega_{\pi_2}=\chi^\sigma\chi^{-1}$ for a character $\chi$ of $E^\times$. Due to \cite[Theorem 2.5.2]{lu2016new}, there exist two representation $\rho_i$ of $\GL_2(F)$ such that $\phi_{\rho_i}|_{WD_E}=\phi_{\pi_i\otimes\chi}$ and $\det\phi_{\rho_1}=\det\phi_{\rho_2}$.
	Then $\phi=\phi_{\rho_1}^\vee\otimes\phi_{\rho_2}$ is what we need.
	\item Assume that $\phi_{\pi}=\phi_{\tau}|_{WD_E}$, where $\tau$ is an irreducible representation of $H(F)$. Choose a representation $\tilde{\tau}$ of $\mathrm{USO}_{2,2}$ such that $\tau\subset \tilde{\tau}|_{\SO_{2,2}(F)}$. Suppose that $\Sigma$ is a representation of $\GSO_{2,2}(E)$ satisfying
	$\phi_{\Sigma}=\phi_{\tilde{\tau}}|_{WD_E}$. Due to Theorem \ref{GSO(2,2)}, $\Sigma$ is $H(F)$-distinguished and $\Sigma|_{\SO_{2,2}(E)}\supset\pi$.
	Thanks to \cite[Proposition 3.1]{anandavardhanan2003distinguished},
	$\pi$ is $H(F)$-distinguished if and only if $\pi$ is $\psi_0$-generic. Then we are done.
	
	\item We separate them into several cases:
	\begin{enumerate}[(A)]
		\item 
	 Suppose that the  representation $\pi$ is square-integrable. Then $\pi$ is $\SO_{2,2}(F)$-distinguished if and only if $\pi$ is $\SO_{4,0}(F)$-distinguished. Assume that $\Sigma=\pi_1\boxtimes\pi_2$ is a $\GSO_{2,2}(F)$-distinguished representation of $\GSO_{2,2}(E)$ and $\Sigma|_{\SO_{2,2}(E)}\supset\pi$. Then each $\pi_i$ is $\GL_2(F)$-distinguished. Set
	$$X_\Sigma=\{\chi:F^\times\rightarrow\mathbb{C}^\times|\Hom_{\GSO_{2,2}(F)}(\Sigma,\chi\circ\lambda_V)\neq0  \},$$ $Y_\Sigma=\{\chi:E^\times\rightarrow\mathbb{C}^\times|\Sigma\otimes\chi=\Sigma \}$ and $Z_\Sigma=\{\chi\in Y_{\Sigma}|\chi|_{F^\times}=\mathbf{1} \}$. Note that there is a bijection between $X_\Sigma$ and $Z_{\Sigma}$. Moreover, following \cite[Lemma 22]{prasad2015arelative}, one has
	\[\dim\Hom_{H(F)}(\pi,\mathbb{C})=\frac{|X_\Sigma|}{|Y_\Sigma|/|Z_{\Sigma}|}. \]
	\begin{enumerate}[label=(A\arabic*)]
		\item If $\pi_1=\pi_2\otimes\chi$ with $\chi$ necessarily quadratic, then $\phi_{\pi}=\chi Ad(\phi_{\pi_1})\oplus\chi$ and $\chi|_{F^\times}=\mathbf{1}$. Hence there exists a quadratic character $\chi_F$ of $W_F$ such that $\chi_F|_{W_E}=\chi$. Note that $\Hom_{\SL_2(F)}(\pi_1,\mathbb{C})\neq0 $. 
		Due to \cite[Theorem 1.2]{L2018pacific}, there exists an identity
		\[\frac{|X_\Sigma|}{|Y_\Sigma|/|Z_\Sigma|}=|\{\phi:WD_F\longrightarrow\PGL_2(\mathbb{C})|\phi|_{WD_E}=Ad(\phi_{\pi_1}) \}|=|\{\phi_1,\phi_2,\cdots,\phi_r \}|=r \]
		where $r=1,2$ or $4$ and so $$F(\phi_{\pi})=\bigcup_{i=1}^r\{\chi_F(\phi_i\oplus\mathbb{C}),\chi_F\omega_{E/F}(\phi_i\oplus\mathbb{C})\}.$$ 
		Hence
		$|F(\phi_{\pi})|=2\cdot\dim\Hom_{H(F)}(\pi,\mathbb{C})$.
		\item If $\pi_1\neq\pi_2\otimes\chi$, then $|Y_\Sigma|=|Z_\Sigma|=1$ or $2$. Suppose that $\phi\in F(\phi_{\pi})$. Then
		\[F(\phi_{\pi})=\bigcup_{\chi_i\in X_\Sigma} \{\phi\otimes\chi_i,\phi\otimes\chi_i\omega_{E/F} \}. \]
		Inside the $L$-packet $\Pi_{\phi_{\pi}}$, only one member is $\SO_{2,2}(F)$-distinguished since $|Y_{\Sigma}|=|Z_{\Sigma}|$.
	\end{enumerate}
\item 
If $\Sigma$ is not square-integrable and $\pi_1=\pi(\chi_1,\chi_2)$ with $\chi_1\neq\chi_2$ and $\chi_1|_{F^\times}=\chi_2|_{F^\times}=\mathbf{1}$, then $$\phi=\phi\otimes\omega_{E/F}$$ where $\phi=\phi_{\rho_1}^\vee\otimes\phi_{\rho_2}$ since $\phi_{\rho_1}=\phi_{\rho_1}\otimes\omega_{E/F}$. In this case, $\pi$ is distinguished by $\SO_{2,2}(F)$, not distinguished by $\SO_{4,0}(F)$ unless $\chi_1=\chi_2\cdot \chi_F\circ N_{E/F}$ with $\chi_F^2=\mathbf{1}$. If it happens, then 
\[\frac{|Y_\Sigma|}{|Z_\Sigma|}\cdot\dim\Hom_{\SO_{4,0}(F)}(\pi,\mathbb{C})=|\{\chi:F^\times\rightarrow\mathbb{C}^\times|\Hom_{\GSO_{4,0}(F)}(\Sigma,\chi\circ\lambda_V)\neq0 \} |. \]
Suppose that $\chi_1=\chi_2\cdot \chi_F\circ N_{E/F}$ with $\chi_F^2=\mathbf{1}$.
\begin{itemize}
	\item If $\pi_2$ is a square-integrable representation, then $X_{\Sigma}=\{\mathbf{1},\chi_F \}$ and
	\[\dim\Hom_{\SO_{2,2}(F)}(\pi,\mathbb{C})=\begin{cases}
	2,&\mbox{ if }\pi_2\mbox{ is dihedral with respect to }\chi_F\circ N_{E/F};\\
	1,&\mbox{ otherwise.}
	\end{cases}  \]
	If $\pi_2=\pi_2\otimes\chi_F\circ N_{E/F}$, then $\dim\Hom_{\SO_{4,0}(F)}(\pi,\mathbb{C})=1 $. On the Galois side, $$(\chi_1^{-1}\otimes\pi_2)^\sigma=\chi_1\otimes\pi_2^\sigma\cong\chi_1\otimes\pi_2^\vee\cong\chi_2^{-1}\otimes\pi_2\cong\chi_1^{-1}\otimes\pi_2. $$
	There exists a parameter $\rho:WD_F\longrightarrow \GL_2(\mathbb{C})$ such that $\rho|_{WD_E}=\chi_1^{-1}\otimes\phi_{\pi_2}$. Thus
	\[F(\phi_{\pi})=\{\phi=\phi_{\rho_1}^\vee\otimes\phi_{\rho_2}\}\cup \bigcup_{z=0}^1 \{(\rho\oplus\rho^\vee)\omega_{E/F}^z \}.\]
	Note that $\phi=\phi\otimes\omega_{E/F}$; then it picks up the trivial pure inner form.
	%\[(\rho\oplus\rho^\vee)|_{WD_E}=\chi_1^{-1}\nu\phi_{\pi_2}\oplus (\chi_2\nu)^{-1}\phi_{\pi_2}. \]
	\item If $\pi_2$ is reducible, then
	\[\dim\Hom_{\SO_{2,2}(F)}(\pi,\mathbb{C})=\begin{cases}
	3,&\mbox{ if }\pi_2=\pi(\chi,\chi\cdot\chi_F\circ N_{E/F}),\chi|_{F^\times}=\mathbf{1};\\
	1,&\mbox{ otherwise.}
	\end{cases}  \]
	If $\pi_2=\pi(\chi,\chi\cdot \chi_F\circ N_{E/F})$ with $\chi|_{F^\times}=\mathbf{1}$, then $\dim\Hom_{\SO_{4,0}(F)}(\pi,\mathbb{C}) =2$. Note that $\chi_1^2=\chi^2$. Then $\chi\chi_1^{-1}$ is a $\Gal(E/F)$-invariant quadratic character. Suppose that $\chi\chi_1^{-1}=\chi_F'\circ N_{E/F}$ with $(\chi_F')^2=\mathbf{1}$. There are $5$ parameter lifts for 
	\[\phi_{\pi}=2\chi_F'|_{W_E}\oplus 2(\chi_F\chi_F')|_{W_E}. \]
	Inside the finite set $F(\phi_{\pi})$, there exists a discrete parameter for $\SO_{2,2}(F)$, i.e.
	\[\phi=\chi_F'(\mathbb{C}\oplus Ad(Ind_{W_E}^{W_F}\nu)) \]
	with $\nu^\sigma=\nu\cdot\chi_F\circ N_{E/F}$, which picks up the trivial pure inner form in $H^1(F,\SO_{2,2})$.
\end{itemize}
\item 
 If neither $\pi_1$ nor $\pi_2$ is of the form 
$\pi(\chi_1,\chi_2)$ with $\chi_1\neq\chi_2$ and $\chi_1|_{F^\times}=\chi_2|_{F^\times}=\mathbf{1}$, then $\pi_1=\pi(\mu_1^{-1},\mu_1^\sigma)$ and $\pi_2=\pi(\mu_2^{-1},\mu_2^\sigma)$ or $\pi_2$ is a square-integrable representation.
\begin{enumerate}[label=(C\arabic*)]
	\item If $\pi_2$ is square-integrable, then $\dim\Hom_{H(F)}(\pi,\mathbb{C})=1$ and $$F(\phi_{\pi})=\{\phi_{\rho_1}^\vee\otimes\phi_{\rho_2},\omega_{E/F}\phi_{\rho_1}^\vee\otimes\phi_{\rho_2} \}.$$
	\item If $\pi_2=\pi(\mu_2^{-1},\mu_2^\sigma)$, then $\mu_1\mu_2^\sigma=\mu_1^\sigma\mu_2$. There are several subcases:
	\begin{enumerate}[label=(C2\alph*)]
		\item If either $\mu_1^\sigma\mu_1\neq\mu_2^\sigma\mu_2$, or $\mu_1^\sigma\mu_1=\mathbf{1}=\mu_2^\sigma\mu_2$, or $\mu_1^\sigma\mu_1=\mu_2^\sigma\mu_2$ is not a quadratic character, then $\Sigma|_{\SO_{2,2}(E)}$ is irreducible and $$\dim\Hom_{\SO_{2,2}(F)}(\pi,\mathbb{C})=2=\dim\Hom_{\SO_{4,0}(F)}(\pi,\mathbb{C}). $$
		\item If $\mu_1^\sigma\mu_1=\mu_2^\sigma\mu_2=\chi_F\circ N_{E/F}$ with $\chi_F^2=\omega_{E/F}$, then $\dim\Hom_{H(F)}(\pi,\mathbb{C})=1$. On the Galois side,
		$$\phi_{\pi}=\mu_1\mu_2^\sigma((\mathbb{C}\oplus\chi_F|_{W_E})\otimes(\mathbb{C}\oplus\chi_F|_{W_E}) )=\mu_1\mu_2^\sigma(2\chi_F|_{W_E}\oplus2\mathbb{C}).$$ There is a quadratic character $\chi$ of $W_F$ such that $\chi|_{W_E}=\mu_1\mu_2^\sigma$. Then \[F(\phi_{\pi})=\{\chi\otimes(\chi_F^{-1}\oplus\mathbb{C}\oplus\chi_F\oplus\mathbb{C}),\chi\omega_{E/F}\otimes(\chi_F^{-1}\oplus\mathbb{C}\oplus\chi_F\oplus\mathbb{C}) \}. \]
		\item If $\mu_1^\sigma\mu_1=\mu_2^\sigma\mu_2=\chi_F\circ N_{E/F}\neq\mathbf{1}$ with $\chi_F^2=\mathbf{1}$, then $$\dim\Hom_{H(F)}(\pi,\mathbb{C})=\begin{cases}
		2,&\mbox{if }\mu_1\neq\mu_2;\\3,&\mbox{if }\mu_1=\mu_2.
		\end{cases}$$
		If $\mu_1=\mu_2\cdot\chi\circ N_{E/F}$ with $\chi^2=\omega_{E/F}$, then $\dim\Hom_{\SO_{4,0}(F)}(\pi,\mathbb{C}) =2$ and
		\[\phi_{\pi}=\chi|_{W_E}\otimes(2\chi_F|_{W_E}\oplus2\mathbb{C}). \]
		In this case, $F(\phi_{\pi})=\{\phi=\chi\otimes(\chi_F\omega_{E/F}\oplus\mathbb{C}\oplus\chi_F\oplus\omega_{E/F}) \}$ is a singleton
		and $\deg\Phi(\phi)=4$. 
		%Near the parameter $\phi$, the base change map $\Phi$ is equivalent to
		%\[\mathbb{C}\times\mathbb{C}\longrightarrow \mathbb{C}/\mu_2\times\mathbb{C}/\mu_2. \]
	%	\[\chi|-|_F^{z_1}\oplus\chi^{-1}|-|_F^{-z_1}\oplus\chi\chi_F|-|_F^{z_2}\oplus\chi^{-1}\chi_F|-|_F^{-z_2}\mapsto \chi|_{W_E}(|-|^{z_1}_E\oplus|-|_E^{-z_1})\oplus\chi\chi_F|_{W_E}(|-|_E^{z_2}\oplus|-|_E^{-z_2})  \]
%where $|-|_F$ (resp. $|-|_E$)	is the absolute value function of $F^\times$ (resp. $E^\times$), $z_i\in\mathbb{C}$.
\end{enumerate}
If $\mu_1=\mu_2$ and $\mu_1^\sigma\mu_1=\chi_F\circ N_{E/F}$, then $\dim\Hom_{\SO_{4,0}(F)}(\pi,\mathbb{C}) =2$. On the Galois side, there are four obvious lifts for $\phi_{\pi}$: $2\chi_F\oplus2\mathbb{C}$, $2\chi_F\oplus2\omega_{E/F}$, $2\chi_F\omega_{E/F}\oplus2\mathbb{C}$ and $2\chi_F\omega_{E/F}\oplus2\omega_{E/F}$. Suppose that $\nu^\sigma=\nu\cdot\chi_F\circ N_{E/F}$. Then $\chi_F\otimes(\mathbb{C}\oplus Ad(Ind_{W_E}^{W_F}\nu))$ lies in $F(\phi_{\pi})$ and it picks up the trivial pure inner form. 
\end{enumerate}
	\end{enumerate}
Then we have finished the proof.
\end{enumerate}
\end{proof}

\section{The Prasad conjecture for $\Sp_4$}
This section focuses on the proof for Theorem \ref{sp(4)}.
Let us recall the Prasad conjecture for $\GSp(4)$ over a quadratic field extension $E/F$. If $H=\GSp_4$, then $\chi_{H}=\omega_{E/F}$ and $H^{op}=\mathrm{USp}_4$, where
\[H^{op}(F)=\{g\in\GSp_4(E)| \lambda_W(g)^{-1}g=\sigma(g) \} .\]
\begin{thm}\cite[Theorem 3.4.6]{lu2018GSp(4)}
	Let $E$ be a quadratic field extension over a nonarchimedean local field $F$. Given a tempered representation $\tilde{\tau}$ of $\GSp_4(E)$ with an enhanced $L$-parameter $(\phi_{\tilde{\tau}},\lambda)$, then
	\[\dim\Hom_{\GSp_4(F)}(\tilde{\tau},\omega_{E/F})=\sum_{\phi\in F(\phi_{\tilde{\tau}}) }m(\lambda,\phi)\frac{\deg\Phi(\phi)  }{d_0(\phi)}   \]
	where
	\begin{itemize}
		\item $F(\phi_{\tilde{\tau}} )=\{\phi:WD_F\longrightarrow{}^LH^{op}|\phi|_{WD_E}=\phi_{\tilde{\tau}}  \}$;
		\item $m(\lambda,\phi)$ is the multiplicity of the trivial character of $S_\phi$ contained in $\lambda|_{S_{\phi}}$;
		\item  $d_0(\phi)=|Coker\{S_{\phi}\longrightarrow S_{\phi_{\tilde{\tau}} }^{\Gal(E/F)} \} |$.
	\end{itemize}
In particular, if $\tilde{\tau}$ is tempered and nongeneric, then $\Hom_{\GSp_4(F)}(\tilde{\tau},\omega_{E/F})=0$.
\end{thm}
\begin{rem}
	We will say that the Langlands parameter $\phi_{\tilde{\tau}}$ is conjugate-orthogonal if the composite
	\[i\circ\phi_{\tilde{\tau}}:WD_E\longrightarrow\GL_4(\mathbb{C})  \]
	is conjugate-orthogonal in the sense of \cite[\S 3]{gan2011symplectic}, where $i:\GSp_4(\mathbb{C})\longrightarrow\GL_4(\mathbb{C})$ is the natural embedding. Then the generic tempered representation $\tilde{\tau}$ of $\GSp_4(E)$ is $\GSp_4(F)$-distinguished if and only if $\phi_{\tilde{\tau}}$ is conjugate-orthogonal. We will identify the characters of $F^\times$ and the characters of $W_F$ by the local class field theory.
\end{rem}

\begin{thm}\cite[Theorem 3.3.8]{lu2018GSp(4)}
	Let $\tilde{\tau}$ be a generic tempered representation of $\GSp_4(E)$, distinguished by $\GSp_4(F)$. Set $\phi_0=\chi_1\oplus\chi_2$ with $\chi_1\neq\chi_2$ and $\chi_1|_{F^\times}=\chi_2|_{F^\times}=\mathbf{1}$. Then
	\begin{equation}\label{multiplicity2}
		\dim\Hom_{\GSp_4(F)}(\tilde{\tau},\mathbb{C})=\begin{cases}
	2,&\mbox{ if }\phi_{\tilde{\tau}}=\phi_1\oplus\phi_2(\mbox{endoscopic case}) \mbox{ with }\phi_i\neq\phi_0\mbox{ conjugate-orthogonal};\\
	1,&\mbox{ otherwise.} 	\end{cases}  
	\end{equation}
\end{thm}
\begin{rem}
	The identity \eqref{multiplicity2} could be very useful when we study the multiplicity $\dim\Hom_{\Sp_4(F)}(\tau,\mathbb{C}) $ for a generic representation $\tau$ of $\Sp_4(E)$.
\end{rem}
Suppose that $\tau$ is a $\Sp_4(F)$-distinguished representation of $\Sp_4(E)$. Let $\tilde{\tau}$ be
a $\GSp_4(F)$-disitinguished representation of $\GSp_4(E)$ and $\tilde{\tau}|_{\Sp_4(E)}\supset \tau$. Fix a nontrivial additive character $\psi_0(e)=\psi(tr_{E/F}(\delta e))$ for $e\in E$ where $E=F[\delta]$
with $\delta^2\in F^\times\setminus {F^\times}^2$.
Set \[\GSp_4(E)^\natural=\{g\in\GSp_4(E)|\lambda_W(g)\in F^\times \}. \]
Let $\pi$ be a generic smooth representation of $\GSp_4(E)$. Consider the restriction of $\pi$ to $\GSp_4(E)^\natural$ and write it as a direct sum of irreducible representations:
\[\pi|_{\GSp_4(E)^\natural}=\pi^+\oplus \oplus_{i\in I} \pi_i \]
where $\pi^+$ is $\psi_0$-generic and $\pi_i$'s are not $\psi_0$-generic. Note that the finite set $I$ may possibly be empty.

 Let $\pi^0$ be a representation of $\GSp_4(E)^\natural$, define
\[(\pi^0)^\delta(g)=\pi^0(g_\delta^{-1} g g_\delta) \]
for $g\in\GSp_4(E)$, where $g_\delta\in\GSp_4(E)$ with similitude $\lambda_W(g_\delta)=\delta$.
\begin{prop}\cite[Theorem 1.1]{anandavardhanan2003distinguished}
	Let $\pi$ be an irreducible, smooth, generic and tempered representation of $\GSp_4(E)$. Then
	$\pi$ is $\Sp_4(F)$-distinguished if and only if $\pi^+$ is $\Sp_4(F)$-distinguished.\label{genericrep}
\end{prop}
\begin{proof}
	It suffices to show that if $(\pi^0)^\delta$ is a $\Sp_4(F)$-distinguished representation of $\GSp_4(E)^\natural$, then $(\pi^0)^\delta$ has a Whittaker model with respect to a nontrivial additive character of $N(E)$ which is trivial on $N(F)$, where $N\subset\Sp_4$ is a unipotent subgroup.
	Without loss of generality, we may assume that $(\pi^0 )^\delta$ is $\GSp_4(F)$-distinguished.
	Due to \cite[Theorem 4.2.18]{lu2016new},
	\[\dim\Hom_{\GSp_4(F)}((\pi^0)^\delta,\mathbb{C} )=\dim\Hom_{\GO_{3,3}(F)} (\Theta_{\psi_E}^{3,3}(\pi^0),\mathbb{C})+\dim\Hom_{\GO_{4,0}(F) }(\Theta_{\psi_E}^{2,2}(\pi^0),\mathbb{C})    \]
	where $\Theta_{\psi_E}^{n,n}(\pi^0)$ is the big theta lift from $\GSp_4(E)^\natural$ to $\GO_{n,n}(E)^\natural$ for $n=2,3$. (See \cite[\S4.2]{lu2016new} for more details.) Then  $\Theta_{\psi_E}^{3,3}(\pi^0)$ is $\GO_{3,3}(F)$-distinguished. Moreover, $\Theta_{\psi_E}^{3,3}(\pi^0)|_{\GSO_{3,3}(E)^\natural}$ is irreducible and is $\GSO_{3,3}(F)$-distinguished. Suppose that $\Theta_{\psi_E}^{3,3}(\pi^0)|_{\GSO_{3,3}(E)^\natural }=\Pi\boxtimes\chi$, where $\Pi$ is an irreducible representation of $$\GL_4(E)^\natural=\{g\in\GL_4(E)|\det(g)\in F^\times \}.$$ Due to Lemma \ref{GL(4)} (below), the representation $\Theta^{3,3}_{\psi_E}(\pi^0)$ is $\psi_0$-generic and so is its theta lift $\theta_{\psi_0}(\theta_{\psi_E}^{3,3}(\pi^0))=(\pi^0)^\delta $, i.e. $(\pi^0)^\delta$ has a Whittaker model with respect to $\psi_0$.
\end{proof}
\begin{rem}
	Proposition \ref{genericrep} is the key statement used in \cite{anandavardhanan2003distinguished,anandavardhanan2016distinguished} to verify the Prasad conjecture for $\SL_n$. It will be the key result used in this paper for the proof of Theorem \ref{sp(4)} as well.
\end{rem}
Set $\GL_4(E)^\natural=\{g\in\GL_4(E)|\det(g)\in F^\times \}$ and let $\mathfrak{N}\subset \GL_4$ be a unipotent subgroup.
\begin{lem}\cite[Lemma 4.2]{anandavardhanan2016distinguished}\label{GL(4)}
	Let $\Pi$ be a generic tempered representation of $\GL_4(E)^\natural$, distinguished by $\GL_4(F)$. Then $\Pi$ has a Whittaker model for a nondegenerate character $\psi_0:\mathfrak{N}(E)/\mathfrak{N}(F)\longrightarrow\mathbb{C}^\times$.
\end{lem}
It follows from the fact that $l\in\Hom_{\GL_4(F)}(\Pi,\mathbb{C})$ can be written uniquely (up to a scalar) as
\[W\mapsto \int_{\mathfrak{N}(F)\backslash P_1(F)}W(pg)dp  \]
where $P_1(F)$ is the mirabolic subgroup of $\GL_4(F)$, the subgroup of $\GL_4(F)$ with last row $(0,\cdots, 0,1)$ and $W$ is a Whittaker function in $\Pi$. (See \cite[Theorem 1.1]{matringe2017test}.)
\begin{lem}
	Let $\pi$ be a $\GL_2(F)$-distinguished supercuspidal representation of $\GL_2(E)$. If there exists a unique nontrivial quadratic character $\chi$ such that $\pi\otimes\chi\cong\pi$, then $\chi|_{F^\times}=\mathbf{1}$.
\end{lem}
\begin{proof}
	If $\chi|_{F^\times}\neq\mathbf{1}$, then $\Hom_{\GL_2(F)}(\pi,\chi|_{F^\times})\neq0 $. Since $\pi$ is supercuspidal, $\Hom_{\GL_2(F)}(\pi,\omega_{E/F})=0 $ and $\chi|_{F^\times}\neq\omega_{E/F}$. Hence $(\pi\otimes\chi)^\sigma=(\pi\otimes\chi)^\vee$. Note that $\pi^\vee=\pi^\sigma$. Then
	\[\pi^\vee=\pi^\vee\otimes\chi^\sigma\chi \]
	and $\chi^\sigma\chi\neq\mathbf{1}$, which contradicts the assumption that there exists a unique character $\chi$ satisfying $\pi\otimes\chi=\pi$.
\end{proof}
\begin{lem}
	Let $\pi$ be a $\GL_2(F)$-distinguished supercuspidal representation of $\GL_2(E)$. If $\pi\otimes\chi\cong\pi$ and $\chi|_{F^\times}=\chi_F\neq\mathbf{1}$, then $\pi\otimes\chi_F\circ N_{E/F}\cong \pi$.
\end{lem}
\begin{proof}
	Since $\pi$ is $\GL_2(F)$-distinguished,  the $L$-parameter $\phi_{\pi}$ is conjugate-orthogonal. So is $\phi_\pi\otimes\chi$. Then $\phi_{\pi}^\vee=\phi_{\pi^\sigma}$ and
	$(\phi_{\pi})^\vee\otimes\chi=(\phi_{\pi^\sigma})\otimes\chi^\sigma$. So
	\[\phi_{\pi}=\phi_{\pi}\otimes \chi\chi^\sigma=\phi_{\pi}\otimes \chi_F|_{W_E} \]
	which means that $\pi\cong\pi\otimes \chi_F\circ N_{E/F}$.
	 %$\Hom_{\GL_2(F)}(\pi,\chi_F)=\Hom_{\GL_2(F)}(\pi\otimes\chi,\mathbb{C})=\Hom_{\GL_2(F)}(\pi,\mathbb{C})\neq0  $, there exists a character $1$
\end{proof}
\begin{rem}
	Suppose that $\tau$ is a $\SL_2(F)$-distinguished representation of $\SL_2(E)$ and $$\phi_{\tau}=Ad(\phi_{\pi})=\chi\oplus \chi_F|_{W_E}\oplus\chi\chi_F|_{W_E}.$$ Then there exists only one lifted parameter
	\[\phi=\chi_F\oplus Ind_{W_{E'}}^{W_F}\chi' \]
	such that $\phi|_{WD_E}=\phi_\tau $, where $E'$ is a quadratic field extension over $F$ with associated quadratic character $\omega_{E'/F}=\chi_F$ and $\chi'$ is a character of $W_{E'}$. (See \cite[Page 491]{L2018pacific}.) There are $4$ elements in the $L$-packet $\Pi_{\phi_{\tau}}$ and half of them are $\SL_2(F)$-distinguished, which is compactible with the fact that the group homomorphism
	\[S_\phi=\mu_2\hookrightarrow S_{\phi_\tau }=\mu_2\times\mu_2 \]
	is the diagonal embedding and two charactes on $S_{\phi_\tau }$ restricted to $S_\phi$ contains the trivial representation. It can be viewed as the first evidence why $m(\lambda,\phi)$ appears in \eqref{conjidentity} in the Prasad conjecture.
\end{rem}

Now we give the proof of Theorem \ref{sp(4)}. Let $H=\Sp_4$, $\chi_H=\mathbf{1}$. Then $H^{op}=\Sp_4$.  Fix $\ell\in  W_F\setminus W_E$. The proof heavily depends on 
\begin{itemize}
	\item the description of the $L$-packets for $\Sp_4$ in \cite[\S6]{takeda2010Sp(4)} and
	\item the Prasad conjecture for $\SL_2$ proved in \cite[Theorem 1.2]{L2018pacific}.
\end{itemize}
We will use the restriction to $\Sp_4(E)$ from $\GSp_4(F)$-distinguished representations of $\GSp_4(E)$ to study the multiplicity $\dim\Hom_{\Sp_4(F)}(\tau,\mathbb{C}) $, following the methods used by Anandavardhanan and Prasad in \cite{anandavardhanan2003distinguished,anandavardhanan2016distinguished}. Compared with the case $H=\SL_2$,  the $L$-packets of $\Sp_4$ are much more complicated  and so the proof looks much more difficult. However the ideas are essentially the same.
%\subsection*{The proof of Theorem \ref{sp(4)}} 
\begin{proof}[Proof of Theorem \ref{sp(4)}]
Assume that $\tau$ is a tempered representation of $\Sp_4(E)$ with an enhanced $L$-parameter $(\phi_\tau,\lambda )$, distinguished by $\Sp_4(F)$. Suppose that $\tilde{\tau}$ is an irreducible tempered  $\GSp_4(F)$-distinguished representation of $\GSp_4(E)$ such that $\tilde{\tau}|_{\Sp_4(E)}\supset\tau$.
Let $$X_{\tilde{\tau}}=\{\chi:F^\times\longrightarrow\mathbb{C}^\times|\Hom_{\GSp_4(F)}(\tilde{\tau},\chi\circ\lambda_W )\neq0  \},$$ $Y_{\tilde{\tau}}=\{\chi:E^\times\longrightarrow\mathbb{C}^\times|\phi_{\tilde{\tau}}\otimes\chi=\phi_{\tilde{\tau}}  \}$
and $Z_{\tilde{\tau}}=\{\chi\in Y|\chi|_{F^\times}=\mathbf{1} \}$. Then $\phi_{\tilde{\tau}}^\vee=\phi_{\tilde{\tau}^\sigma}$ and
\[\frac{|Y_{\tilde{\tau}}|}{|Z_{\tilde{\tau}}|} \dim\Hom_{\Sp_4(F)}(\tau,\mathbb{C})=\sum_{\chi\in X_{\tilde{\tau}} }\dim\Hom_{\GSp_4(F)}(\tilde{\tau},\chi\circ\lambda_W).   \]
According to the Langlands parameter $\phi_{\tilde{\tau}} $,
the proof is divided as $3$  parts: 
%\begin{enumerate}[(A).]
\begin{itemize}
		\item  the Langlands parameter $\phi_{\tilde{\tau}} $ is irreducible (see \ref{A});
	\item the Langlands parameter $\phi_{\tilde{\tau}} =\rho\oplus\chi\rho$ with $\chi\neq\mathbf{1}$ is reducible (see \ref{B});
	\item the endoscopic case (see \ref{C}).
\end{itemize}
%\end{enumerate}
We will study the multiplicity $\dim\Hom_{\Sp_4(F)}(\tau,\mathbb{C}) $ case by case.
\begin{enumerate}[(A)]
	\item \label{A} If $\phi_{\tilde{\tau}}$ is irreducible, then there exists a bijection between the set $X_{\tilde{\tau}}$ and $Z_{\tilde{\tau}}$ and so
	\[\dim\Hom_{\Sp_4(F)}(\tau,\mathbb{C})=\frac{|X_{\tilde{\tau}}|}{|Y_{\tilde{\tau}}|/|Z_{\tilde{\tau}}| }=\frac{|Z_{\tilde{\tau}}|^2}{|Y_{\tilde{\tau}}|}.  \]
%\begin{enumerate}[(a).]
	%[label*=\arabic*]
	\begin{enumerate}[label=(A\arabic*)]
		\item If $\phi_{\tilde{\tau}}$ is primitive, then $p=2$ and $ \phi_{\tau}^\vee\cong\phi_{\tau}\cong\phi_{\tau^\sigma}  $. Moreover, there exists $A\in\SO_5(\mathbb{C})$ such that $\phi_{\tau^\sigma}(t)=\phi_{\tau}^\ell(t)=\phi_{\tau}(\ell t\ell^{-1})=A\phi_{\tau}(t)A^{-1} $ for $t\in W_E$ and $\phi_{\tau}(\ell^2) =A^2$. Set $\phi(\ell)=A$ and $\phi(t)=\phi_{\tau}(t) $ for $t\in W_E$. Then $\phi\in F(\phi_{\tau} )$ and $F(\phi_{\tau} )=\{\phi\}$ is a singleton since $\phi_{\tau} $ is irreducible. 

In general,	if $|Y_{\tilde{\tau}}|=|Z_{\tilde{\tau}}|$, then
\[\dim\Hom_{\Sp_4(F)}(\tau,\mathbb{C})=|Z_{\tilde{\tau}}|.  \]
	%	 If $\dim\Hom_{\Sp_4(F)}(\tau,\mathbb{C})=4 $ which may happen only when $p=2$, then 
		 %there exists an isomorphism $\Gal(L/F)\cong\mu_2\times\mu_2\times\mu_2$. Moreover, 
In this case,		 there exist  characters $$\nu_i:W_E\longrightarrow\mathbb{C}^\times \quad(i=1,2,\cdots,4)$$ such that $\nu_i^\sigma\nu_i^{-1}=sim(\phi_{\tilde{\tau}})\cdot\chi_i$ for $\chi_i\in Z_{\tilde{\tau}}$. Then $$(\phi_{\tilde{\tau}}\otimes\nu_i)^\ell\cong \phi_{\tilde{\tau}}^\vee\otimes\nu_i^\sigma\cong\phi_{\tilde{\tau}}\otimes\nu_i\chi_i\cong\phi_{\tilde{\tau}}\otimes\nu_i.$$
Here we identify the character $\nu_i^\ell$ of $W_E$ and $\mu_i^\sigma$ of $E^\times$ by the local class field theory.
		Thus for each $i$, there exists $A_i\in\GSp_4(\mathbb{C})$ such that $$\phi_{\tilde{\tau}}(\ell g\ell^{-1})\nu_i(\ell g\ell^{-1})=A\phi_{\tilde{\tau}}(g)A^{-1}\nu_i(g)$$ for $g\in WD_E$. There exists $\tilde{\phi}_i:WD_F\rightarrow\GSp_4(\mathbb{C})$ such that $$\tilde{\phi}_i|_{WD_E}=\phi_{\tilde{\tau}}\otimes\nu_i $$ for each $i$. Therefore
		\[F(\phi_{\tau})=\bigcup_{i=1}^{|Z_{\tilde{\tau}}|}\{p\circ \tilde{\phi}_i \}. \]
	%	\[As(\rho)|_{W_M}=Ind_{W_L}^{W_M}\chi\chi^s\oplus Ind_{W_L}^{W_M}\chi\chi^{st} \]
%	\end{itemize}
%\noindent(c2)
%If $\rho^s=\rho\otimes\chi$ with $\chi^s\neq\chi$, then we can use the same method. Set $$sim(\phi_{\tau} )\chi_i=\nu_i^\sigma\nu_i^{-1}$$ for each $\chi_i\in Z_{\tilde{\tau}}$, then there exists $\tilde{\phi}_i:WD_E\longrightarrow\GSp_4(\mathbb{C})$ such that $\tilde{\phi}_i|_{WD_E}=\phi_{\tilde{\tau}}\otimes\nu_i $ and $F(\phi_\tau )=\{p\circ \tilde{\phi}_i,i=1,2,\cdots,r \}$ with $|Z_{\tilde{\tau}}|=r$.
%	\[\phi_{\tau}=\omega_{M/E}\oplus\omega_{M'/E}\oplus\omega_{M/E}\omega_{M'/E}\oplus \phi_0  \]
%where $M'\neq M$ is a quadratic field extension over $E$ contained in $L$ and $\phi_0:W_E\longrightarrow\SO_2(\mathbb{C})$ is irreducible. Assume  $\dim\Hom_{\Sp_4(F)}(\tau,\mathbb{C})=4 $. Thanks to \cite[Theorem 1.2]{L2018pacific}, there are four distinct parameter lifts $$Ad(\tilde{\rho}):W_F\longrightarrow\SO_3(\mathbb{C})$$ such that
%\[Ad(\tilde{\rho}) |_{W_E}=\omega_{M/E}\oplus\omega_{M'/E}\oplus\omega_{M/E}\omega_{M'/E} .\]
%Thus $|F(\phi_{\tau} )|=4$ in this case.
% However,
%\[\dim\Hom_{\Sp_4(F)}(\tau,\mathbb{C})=\begin{cases}
%4,&\mbox{ if }|Y_{\tilde{\tau}}|=|Z_{\tilde{\tau}}|=4;\\
%1,&\mbox{ if }|Y_{\tilde{\tau}}|=4,|Z_{\tilde{\tau}}|=2.
%\end{cases}  \]

\item 
If $|Y_{\tilde{\tau}}|=2|Z_{\tilde{\tau}}|$, then $\dim\Hom_{\Sp_4(F)}(\tau,\mathbb{C})=\frac{1}{2}|Z_{\tilde{\tau}}| $.
%If $\dim\Hom_{\Sp_4(F)}(\tau,\mathbb{C}) =1$, 
In this case, $$Y_{\tilde{\tau}}\supset \langle\chi_E,\chi_E^\sigma,\chi_F\circ N_{E/F} \rangle$$ and $\chi_E\chi_E^\sigma=\chi_F\circ N_{E/F}=\chi_2\in Z_{\tilde{\tau}}$. Note that $\nu_1^\sigma\nu_1^{-1}=sim(\phi_{\tilde{\tau}} )$,
\[\phi_{\tilde{\tau}}\otimes\nu_2=\phi_{\tilde{\tau}}\otimes\chi_E\nu_2= \phi_{\tilde{\tau}}\otimes\nu_1\cdot\chi_E\nu_2\nu_1^{-1}   \]
and so $$(\chi_E\nu_2\nu_1^{-1})^\sigma=\chi_E^\sigma(\nu_2\nu_1^{-1})^\sigma=\chi_E^\sigma\cdot\chi_F\circ N_{E/F}\cdot\nu_2\nu_1^{-1}=\chi_E\nu_2\nu_1^{-1}.$$
Therefore
 $p\circ\tilde{\phi}_1=p\circ\tilde{\phi}_2=\phi$. There are $|Y_{\tilde{\tau}}|$ elements inside the $L$-packet $\Pi_{\phi_{\tau}}$ and only half of them are $\Sp_4(F)$-distinguished. When $|Y_{\tilde{\tau}}|=4$, the component group
\[S_{\phi}=\mu_2\hookrightarrow S_{\phi_{\tau}}=\mu_2\times\mu_2 \]
is the diagonal embedding.

\item If $|Y_{\tilde{\tau}}|=4|Z_{\tilde{\tau}}|$, then $Y_{\tilde{\tau}}=\langle\chi_E,\chi_E^\sigma,\chi_E',(\chi_E')^\sigma \rangle$ and $\chi_E'|_{F^\times}\neq\chi_E|_{F^\times}$. In this case, $$\dim\Hom_{\Sp_4(F)}(\tau,\mathbb{C})=1 $$
and $|F(\phi_\tau )|=1$. Moreover, $p=2$, $|Y_{\tilde{\tau}}|=16$ and $|Z_{\tilde{\tau}}|=4$. On the component group, $$S_\phi=\mu_2\times\mu_2\rightarrow S_{\phi_\tau}$$ is given by $(x,y)\mapsto (x,x,y,y)$. There are $4$ characters of $S_{\phi_{\tau} }$ restricted to $S_\phi$ containing the trivial character and so there are $4$ representations in $\Pi_{\phi_\tau}$ distinguished by $\Sp_4(F)$. 
	\end{enumerate}

%\noindent(c3)
%Now we suppose $\rho^s=\rho\otimes\chi$ with $\chi^s=\chi$. We may assume that $p=2$ and $\phi_{\tilde{\tau}}=\rho_1\otimes\rho_2 $ with $$Ad(\rho_2)=\omega_{M'/E}\oplus\omega_{M/E}\oplus\omega_{M/E}\omega_{M'/E}$$
%where $M'$ is another biquadratic field extension over $F$.  Set $|Y_{\tilde{\tau}}|=2^z$ where $z=2,3$ or $4$, then $$\dim\Hom_{\Sp_4(F)}(\tau,\mathbb{C})=\begin{cases}
%2^z,&\mbox{ if } |Y_{\tilde{\tau}}|=|Z_{\tilde{\tau}}|;\\
%2^{z-2},&\mbox{ if } |Y_{\tilde{\tau}}|=2|Z_{\tilde{\tau}}|;\\
%1,&\mbox{ if }|Y_{\tilde{\tau}}|=16,|Z_{\tilde{\tau}}|=4. 
%\end{cases}$$ 
%If $Z_{\tilde{\tau}}=Y_{\tilde{\tau}}$, then only one member inside the $L$-packet $\Pi_{\phi_\tau}$ is $\Sp_4(F)$-distinguished.
%\item If $\phi_{\tau}=S_5$ is the Steinberg representation, then $\dim\Hom_{\Sp_4(F)}(\tau,\mathbb{C}) =1$.
%\item If $\phi_{\tilde{\tau}}=\rho\otimes S_2$ is the generalized Steinberg representation, then $\rho$ is conjugate-symplectic and $$\dim\Hom_{\Sp_4(F)} (\tau,\mathbb{C})=\begin{cases}1,&\mbox{ if }Y_{\tilde{\tau}}\mbox{ is a singleton};\\
%1,&\mbox{ if }|Y_{\tilde{\tau}}|=4,|Z_{\tilde{\tau}}|=2;\\
%2,&\mbox{ if }|Y_{\tilde{\tau}}|=|Z_{\tilde{\tau}}|=2; \\
%4,&\mbox{ if }|Y_{\tilde{\tau}}|=|Z_{\tilde{\tau}}|=4.
%\end{cases}$$
%\end{enumerate}

	\item\label{B} If $\phi_{\tilde{\tau}}=\rho\oplus\rho\chi$ with similitude character $sim(\phi_{\tilde{\tau}} )=\chi\cdot\det\rho$ and $\chi\neq\mathbf{1}$, then $\phi_{\tau}=\chi\oplus\chi^{-1}\oplus Ad(\rho)$. There are several subcases. 
	\begin{enumerate}[label=(B\arabic*)]
		\item 
%	\noindent(B1)
	If $\chi^2=\mathbf{1}$ and $\rho$ is irreducible, then $\phi_\tau=2\chi\oplus Ad(\rho) $.
%	\begin{enumerate}[(a)]
		
		\noindent(a)
		 If $\rho$ is conjugate-orthogonal and irreducible and so $\chi|_{F^\times}=\mathbf{1}$, then $\chi=\chi^\sigma=\chi_F\circ N_{E/F}$.
		\begin{enumerate}[label=(a\arabic*)]
	\item	 If $\rho=\rho\chi$, then $Ad(\rho):WD_E\longrightarrow\SO_3(\mathbb{C})$ is reducible and
		 \[\dim\Hom_{\Sp_4(F)}(\tau,\mathbb{C})=\begin{cases}2,&\mbox{ if }|Y_{\tilde{\tau}}|=4,|Z_{\tilde{\tau}}|=2;\\
		 2\cdot|Z_{\tilde{\tau}}|,&\mbox{ if }|Y_{\tilde{\tau}}|=|Z_{\tilde{\tau}}|.
		 \end{cases}  \]
		 \begin{itemize}
		 	\item 
		 When $\dim\Hom_{\Sp_4(F)}(\tau,\mathbb{C})=2 $, there exists a unique lift for $Ad(\rho)$. In fact, $\phi=2\chi_F\oplus Ad(\tilde{\rho})$ due to \cite[Theorem 1.2]{L2018pacific}, where $\tilde{\rho}$ is an $L$-parameter of $\GL_2(F)$ and $Ad(\tilde{\rho})|_{WD_E}=Ad(\rho)$. So \[F(\phi_{\tau})=\{2\chi_F\oplus Ad(\tilde{\rho}), Ad(\tilde{\rho})\oplus2\chi_F\omega_{E/F} \}. \]
		  Half members inside the $L$-packet $\Pi_{\phi_\tau}$ are $\Sp_4(F)$-distinguished.
		  \item 
		  		  When $\dim\Hom_{\Sp_4(F)}(\tau,\mathbb{C})=4 $, there are $2$ lifts for $Ad(\rho)$ and so there are $4$ lifts for $\phi_{\tau}$, i.e. $|F(\phi_{\tau})|=4$. Only one member side the $L$-packet $\Pi_{\phi_{\tau}}$ is $\Sp_4(F)$-distinguished.
		  \item 		  
		  		   When $\dim\Hom_{\Sp_4(F)}(\tau,\mathbb{C}) =8$, there are $4$ lifts of $Ad(\rho)$ and so $|F(\phi_\tau )|=8$.
		 	 \end{itemize}
	\item	If $\rho\neq\rho\chi$ and so $\rho^\ell\neq(\rho\chi)^\vee$, then $|Y_{\tilde{\tau}}|=|Z_{\tilde{\tau}}|=2$ if $\rho$ is primitive. Moreover, $$X_{\tilde{\tau}}=\{\mathbf{1},\chi_F,\chi_F\omega_{E/F} \}$$ and
		$\dim\Hom_{\Sp_4(F)}(\tau,\mathbb{C})=3 $. There are $3$ parameter lifts for $\phi_\tau $, i.e.
		\[F(\phi_\tau )=\{2\chi_F\oplus Ad(\tilde{\rho}),2\chi_F\omega_{E/F}\oplus Ad(\tilde{\rho}),\chi_F\oplus\chi_F\omega_{E/F}\oplus\omega_{E/F} Ad(\tilde{\rho}) \}. \]
		 If $\rho$ is dihedral with respect to one nontrivial quadratic character $\chi_E$, then $$\chi_E^\sigma=\chi_E=\chi_F'\circ N_{E/F}$$ with $(\chi_F')^2=\mathbf{1}$ and 
		$\dim\Hom_{\Sp_4(F)}(\tau,\mathbb{C} )= 6 $,
		which only happens when $p=2$. On the Galois side, there are two lifts $\phi_0$ and $\omega_{E/F}\phi_0$ with determinant $\chi_F'$ such that $$\phi_0|_{W_E}\oplus\chi_E=Ad(\rho).$$ 
		It is obvious that $X_{\tilde{\tau}}=\{\mathbf{1},\chi_F,\chi_F',\chi_F\chi_F',\chi_F\omega_{E/F},\chi_F'\chi_F\omega_{E/F} \}$ and
		\[F(\phi_{\tau} )=\bigcup_{z_0,z,z'\in\{0,1\}} \{\chi_F\oplus\chi_F\omega_{E/F}\oplus\omega_{E/F}\chi_F'\oplus\omega_{E/F}^{z_0}\phi_0, 2\chi_F\omega_{E/F}^z\oplus\chi_F'\oplus \omega_{E/F}^{z'}\phi_0 \}. \]
		
	If $\rho$ is dihedral with respect to three nontrivial quadratic characters, then $|Y_{\tilde{\tau}}|=8$ and 
		\[\dim\Hom_{\Sp_4(F)}(\tau,\mathbb{C})=\begin{cases}
		12,&\mbox{ if }Z_{\tilde{\tau}}=Y_{\tilde{\tau}};\\
		3,&\mbox{ otherwise}.
		\end{cases}  \]
		If $Z_{\tilde{\tau}}=Y_{\tilde{\tau}}$, then 
		\[F(\phi_{\tau})=\bigcup_{i=1}^4\bigcup_{z_0\in\{0,1\} } \{\chi_F\oplus\chi_F\omega_{E/F}\oplus\omega_{E/F}Ad(\tilde{\rho}_i),2\chi_F\omega_{E/F}^{z_0}\oplus Ad(\tilde{\rho}_i)\} \]
		where $Ad(\tilde{\rho}_i):WD_F\longrightarrow\SO_3(\mathbb{C})$ are four parameter lifts for $Ad(\rho)$.
	\end{enumerate}
	%	Note that if $z=1$, then $Y_{\tilde{\tau}}$ must equal to $Z_{\tilde{\tau}}$.
	
		\noindent(b) If $(\rho\chi)^\vee=\rho^\ell$ is irreducible, then $\chi^\sigma=\chi$. There exists a character $\nu:W_E\longrightarrow\mathbb{C}^\times$ such that $\nu^\sigma\nu=\chi$ and so $(\rho\otimes\nu)^\ell=(\rho\otimes\nu)^\vee$. Then
		\[\dim\Hom_{\Sp_4(F)}(\tilde{\tau},\mathbb{C})=\dim\Hom_{\Sp_4(F)}(\tilde{\tau}\otimes\nu,\mathbb{C})   \]
		with $\rho\otimes\nu$ conjugate-self-dual, which has been discussed before if $\chi|_{F^\times}=\mathbf{1}$.
		
		 Suppose that $\chi|_{F^\times}\neq\mathbf{1}$. Then $\chi=\chi_F\circ N_{E/F}$ with $\chi_F^2=\omega_{E/F}$  and
		 	\[\dim\Hom_{\Sp_4(F)}(\tau,\mathbb{C})=\begin{cases}
		 |Z_{\tilde{\tau}}|,&\mbox{ if }|Y_{\tilde{\tau}}|=2\cdot|Z_{\tilde{\tau}}|;\\
		 1,&\mbox{ if }|Y_{\tilde{\tau}}|=8,|Z_{\tilde{\tau}}|=2.\\
	 \end{cases}  \]
		 If $|Y_{\tilde{\tau}}|=2$, there exists only one lift $$Ad(\tilde{\rho}):WD_F\longrightarrow\SO_3(\mathbb{C})$$ satisfying $Ad(\tilde{\rho})|_{WD_E}=Ad(\rho)$. So $F(\phi_\tau )$ is a singleton, i.e.
		 \[F(\phi_\tau )=\{\phi=\chi_F\oplus Ad(\tilde{\rho})\oplus \chi_F\omega_{E/F} \} \]
		 and the rest cases are similar. Note that $\deg\Phi(\phi)=2=d_0(\phi)$ in these cases.
%	\end{enumerate}

%\noindent(B2)
\item 
If $\chi^2\neq\mathbf{1}$ and $\rho$ is irreducible, then $\phi_{\tau}=\chi\oplus Ad(\rho)\oplus\chi^{-1} $.
%\begin{itemize}
%\begin{enumerate}[(a)]

	\noindent(c) If $\rho$ is conjugate-orthogonal, then $\chi|_{F^\times}=\mathbf{1}$ and so $\chi\neq\chi^\sigma$. The parameter $\phi_\tau=\chi\oplus\chi^\sigma\oplus Ad(\rho) $ and
	\[\dim\Hom_{\Sp_4(F)}(\tau,\mathbb{C})=\begin{cases}
	|Z_{\tilde{\tau}}|,&\mbox{ if }|Y_{\tilde{\tau}}|=|Z_{\tilde{\tau}}|;\\
	1,&\mbox{ if }|Y_{\tilde{\tau}}|=4,|Z_{\tilde{\tau}}|=2.\\
	\end{cases}  \]
	Thanks to \cite[Theorem 1.2]{L2018pacific}, there are $r$ distinct parameters $Ad(\tilde{\rho}_i):WD_F\rightarrow\SO_3(\mathbb{C})$ such that $Ad(\tilde{\rho}_i)=Ad(\rho)$ where $r=\dim\Hom_{\Sp_4(F)}(\tau,\mathbb{C}) $. Then
	\[F(\phi_{\tau} )=\bigcup_{i=1}^r\{Ind_{WD_E}^{WD_F}\chi\oplus \omega_{E/F}Ad(\tilde{\rho}_i) \}. \]
	
	\noindent(d)
	 If $\rho^\ell=(\rho\chi)^\vee$, then $\chi=\chi^\sigma=\chi_F\circ N_{E/F}$. There exists a character $\nu$ of $E^\times$ such that $\chi=\nu^\sigma\nu$ and $(\rho\otimes\nu)^\ell=(\rho\otimes\nu)^\vee$. Then
	if $|Y_{\tilde{\tau}}|=2$, we have $Z_{\tilde{\tau}}=Y_{\tilde{\tau}}$. Thus
	\[\dim\Hom_{\Sp_4(F)}(\tau,\mathbb{C})=\begin{cases}
	2\cdot |Z_{\tilde{\tau}}|,&\mbox{ if }|Y_{\tilde{\tau}}|=|Z_{\tilde{\tau}}|;\\
	2,&\mbox{ if }|Y_{\tilde{\tau}}|=4,|Z_{\tilde{\tau}}|=2.\\
	\end{cases}   \]
	On the Galois side, $\chi_F^2\neq\omega_{E/F}$ and
	 $$F(\phi_\tau )=\bigcup_{i=1}^r\bigcup_{z\in\{0,1\} }\{\omega_{E/F}^{z}(\chi_F\oplus\chi_F^{-1})\oplus Ad(\tilde{\rho}_i) \}$$ where $r=1,2$ or $4$ depending on $\rho$.
		%\item 
%\end{itemize}
	%with $\rho\otimes\nu$ conjugate-self-dual.
%\end{enumerate}
%Suppose
%	that $\rho=\chi_1\oplus\chi_2$ is reducible, then $\phi_\tau=\chi\oplus\chi^{-1}\oplus\chi_1/\chi_2\oplus\chi_2/\chi_1\oplus\mathbb{C} $.
%If $\rho$ is conjugate-orthogonal or conjugate-symplectic and so $\chi|_{F^\times}=\mathbf{1}$, then
%\[\dim\Hom_{\Sp_4(F)}(\tau,\mathbb{C})=\begin{cases}
%2^{z+1},&\mbox{ if }|Z_{\tilde{\tau}}|=|Y_{\tilde{\tau}}|=2^z;\\
%2^{z-1},&\mbox{ otherwise}.
%\end{cases}  \]
%If $\chi_i^\sigma\chi_i\chi=\mathbf{1}$, then
%
%If $\chi_1^\sigma\chi_2\chi=\mathbf{1}$
%\item If $\rho=\chi_1\oplus\chi_2$ with $\chi_1\neq\chi_2$ and $\chi_1|_{F^\times}=\chi_2|_{F^\times}=\chi|_{F^\times}=\mathbf{1}$, then

%\noindent(B3)
\item 
If $\chi=\mathbf{1}$ or the parameter $\rho$ is reducible, then it belongs to the endoscopic case.
	\end{enumerate}

	\item\label{C} If $\phi_{\tilde{\tau}}=\phi_1\oplus\phi_2$ with similitude character $\det\phi_1=\det\phi_2$ (called the endoscopic case), then there are several subcases. 
\begin{enumerate}[label=(C\arabic*)]
	\item 
	Suppose that $\phi_1:WD_E\longrightarrow\GL_2(\mathbb{C})$ is irreducible.
%	\begin{enumerate}[(a).]
	
		\noindent(a). If $\phi_1=\phi_2$ is irreducible and conjugate-orthogonal, then
		\[\dim\Hom_{\Sp_4(F)}(\tau,\mathbb{C}) =\begin{cases}
		3\cdot|Z_{\tilde{\tau}}|,&\mbox{ if }|Y_{\tilde{\tau}}|=|Z_{\tilde{\tau}}|;\\
		3,&\mbox{ if }|Y_{\tilde{\tau}}|=4,|Z_{\tilde{\tau}}|=2.
		\end{cases}  \] 
		On the Galois side, $\phi_\tau=2\mathbb{C}+Ad(\phi_1) $ and
		\[F(\phi_{\tau})=\bigcup_{i=1}^r\{2\omega_{E/F}\oplus Ad(\tilde{\rho}_i),2\mathbb{C}\oplus Ad(\tilde{\phi}_i),\omega_{E/F}\oplus\mathbb{C}\oplus \omega_{E/F} Ad(\tilde{\rho}_i)\} \]
		where $r=1,2$ or $4$ and each $Ad(\tilde{\rho}_i)$ satisfies $Ad(\tilde{\rho}_i)|_{WD_E}=Ad(\phi_1)$.
		
		\noindent(b).
		 If $\phi_1=\phi_2\otimes\chi\neq\phi_2$ for a quadratic character $\chi$,  then
		\[\phi_{\tau}=\mathbb{C}\oplus\chi\oplus \chi Ad(\phi_1). \]
		\begin{itemize}
			\item If both $\phi_1$ and $\phi_2$ are conjugate-orthogonal, then $\chi|_{F^\times}=\mathbf{1}$ and $\chi=\chi^\sigma=\chi_F\circ N_{E/F}$. So
		\[\dim\Hom_{\Sp_4(F)}(\tau,\mathbb{C})	=\begin{cases}
				2\cdot|Z_{\tilde{\tau}}|,&\mbox{ if }|Y_{\tilde{\tau}}|=|Z_{\tilde{\tau}}|;\\
						4,&\mbox{ if }|Y_{\tilde{\tau}}|=8,|Z_{\tilde{\tau}}|=4.
			\end{cases} \]
			If $|Y_{\tilde{\tau}}|=2$, then $X_{\tilde{\tau}}=\{\mathbf{1},\chi_F,\chi_F\omega_{E/F} \}$. Moreover,
			$\dim\Hom_{\GSp_4(F)}(\tilde{\tau},\mathbb{C})=2$,
			\[\dim\Hom_{\GSp_4(F)}(\tilde{\tau},\chi_F)=1=\dim\Hom_{\GSp_4(F)}(\tilde{\tau},\chi_F\omega_{E/F}),    \]
			in which case, $|F(\phi_\tau )|=4$ and the lifted parameters are given by
			 $$\chi_F\oplus\omega_{E/F}\oplus\chi_F\omega_{E/F} Ad(\tilde{\rho}),\omega_{E/F}\oplus\omega_{E/F}\chi_F+\chi_F Ad(\tilde{\rho})\mbox{ and }\mathbb{C}\oplus \omega_{E/F}^z\chi_F(\mathbb{C}\oplus Ad(\tilde{\rho}))$$ for $z\in\{0,1\}$, 
			where $Ad(\tilde{\rho})|_{WD_E}=Ad(\phi_1)$. The rest cases are similar.
			\item If $\phi_1^\ell=\phi_2^\vee$ and $\chi|_{F^\times}=\mathbf{1}$, then 
			$\phi_2^\ell\otimes\chi=\phi_2^\vee$. There exists a character $\nu$ such that $\chi=\nu\nu^\sigma$ and $(\phi_2\otimes\nu)^\ell=(\phi_2\otimes\nu)^\vee$. Then $$\dim\Hom_{\Sp_4(F)}(\tilde{\tau},\mathbb{C})=\dim\Hom_{\Sp_4(F)}(\tilde{\tau}\otimes\nu,\mathbb{C})  $$
			with $\phi_2\otimes\nu$ conjugate-self-dual, which has been discussed as before.
			\item If $\phi_1^\ell=\phi_2^\vee$ and $\chi|_{F^\times}\neq\mathbf{1}$, then
			$\phi_1^\ell=\phi_1^\vee\otimes\chi$ and so $\phi_1\otimes\chi=\phi_1\otimes\chi^\sigma$.
			%Moreover, $X_{\tilde{\tau}}$ contains			
		%	$\phi_1\otimes\chi_E=\phi_1$ impies $\chi_E\in Z_{\tilde{\tau}}$ and
		%	\[\begin{cases}
	%	\{	\mathbf{1},\omega_{E/F}\},&\mbox{ if }\chi|_{F^\times}=\omega_{E/F};\\
	%	\{\mathbf{1},\omega_{E/F},\chi|_{F^\times},\omega_{E/F}\cdot\chi|_{F^\times} \},&\mbox{ if }\chi|_{F^\times}\neq\omega_{E/F}.
		%	\end{cases} \]
			 Suppose that $\chi|_{F^\times}=\omega_{E/F}$. Then $\chi=\chi^\sigma$. There exists a character $\nu$ such that $\chi=\nu\nu^\sigma$ and $$(\phi_1\otimes\nu)^\ell=(\phi_1\otimes\nu)^\vee.$$ So the parameter $\phi_1\otimes\nu:WD_E\longrightarrow\GL_2(\mathbb{C})$ is conjugate-self-dual. Then $\nu^2\cdot\det\phi_1$ is conjugate-orthogonal. Note that $\det\phi_1$ is conjugate-orthogonal. Then $\nu^2|_{F^\times}=\mathbf{1}$ and so $\chi|_{F^\times}=\mathbf{1}$, which contradicts $\chi|_{F^\times}=\omega_{E/F}$.
			%
%			Therefore
%				\[\dim\Hom_{\Sp_4(F)}(\tau,\mathbb{C})	=\begin{cases}
%			1,&\mbox{ if }|Y_{\tilde{\tau}}|=2,|Z_{\tilde{\tau}}|=1;\\
%			2,&\mbox{ if }|Y_{\tilde{\tau}}|=4,|Z_{\tilde{\tau}}|=2;\\
%			1,&\mbox{ if }|Y_{\tilde{\tau}}|=8,|Z_{\tilde{\tau}}|=2;\\
%			4,&\mbox{ if }|Y_{\tilde{\tau}}|=8,|Z_{\tilde{\tau}}|=4.
%			\end{cases} \]
%				On the Galois side,
%			\[F(\phi_\tau )=\{\omega_{E/F}\oplus  \nu_iAs_{E/F}(\phi_1 ),\nu_i\circ N_{E/F}=\chi_i\in Z_{\tilde{\tau}}  \} \]
%			Note that $tr(\phi_1(\ell^2))=\chi(\ell^2)tr(\phi_1(\ell^2))=0$, then $\omega_{E/F}As_{E/F}(\phi_1)\cong As_{E/F}(\phi_1)$.
			%
			
		Thus $\chi|_{F^\times}\neq\omega_{E/F}$. Then $\phi_1=\phi_1\otimes\chi\chi^\sigma$ and
			\[\dim\Hom_{\Sp_4(F)}(\tau,\mathbb{C})	=\begin{cases}
		|Z_{\tilde{\tau}}|,&\mbox{ if }|Y_{\tilde{\tau}}|=2\cdot|Z_{\tilde{\tau}}|;\\
		1,&\mbox{ if }|Y_{\tilde{\tau}}|=8,|Z_{\tilde{\tau}}|=2..
		\end{cases} \]
		On the Galois side, if $|Y_{\tilde{\tau}}|=2|Z_{\tilde{\tau}}|$, then
		\[F(\phi_\tau )=\{\omega_{E/F}\oplus\nu_i As_{E/F}(\phi_1),\nu_i\circ N_{E/F}\in Z_{\tilde{\tau}} \} .\]
		If $\dim\Hom_{\Sp_4(F)}(\pi,\mathbb{C}) =1$, then $F(\phi_\tau )=\{\omega_{E/F}\oplus As_{E/F}(\phi_1) \}$. Note that $$\omega_{E/F}\otimes As_{E/F}(\phi_1)=As_{E/F}(\phi_1)$$ in these cases. 
		%	where $\tilde{\phi}=Ind_{W_E}^{W_F}(\phi_1\otimes\chi^z\otimes\nu_i)$.
		\end{itemize}
		
		\noindent(c).
		 If $\phi_1\neq\phi_2\otimes\chi$ for any character $\chi$ of $W_E$, then $\phi_\tau=\mathbb{C}\oplus \phi_1^\vee\otimes\phi_2 $.
		\begin{itemize}
			\item If both $\phi_1$ and $\phi_2$ are conjugate-orthogonal and irreducible, then $Y_{\tilde{\tau}}=Z_{\tilde{\tau}}$ and $$\dim\Hom_{\Sp_4(F)}(\tau,\mathbb{C}) =\begin{cases}
			4,&\mbox{ if }|Y_{\tilde{\tau}}|=2;\\
			2,&\mbox{ if }|Y_{\tilde{\tau}}|=1.
			\end{cases}$$
			Suppose that $\det\phi_1=\nu^\sigma\nu^{-1}$, then there exists $\tilde{\rho}_i:WD_F\rightarrow\GL_2(\mathbb{C})$ such that $$\tilde{\rho}_i|_{WD_E}=\phi_i\otimes\nu$$ and $\det\tilde{\rho}_i=\omega_{E/F}\cdot\nu|_{F^\times}$. Then
			$$ F(\phi_\tau  )=\bigcup_{z\in\{0,1\} }\{\mathbb{C}\oplus\omega_{E/F}^z \nu_i \tilde{\rho}_1^\vee\otimes\tilde{\rho}_2,\nu_i\circ N_{E/F}\in Z_{\tilde{\tau} } \}.$$ 
			\item If $\phi_1^\ell=\phi_2^\vee$, then $\dim\Hom_{\Sp_4(F)}(\tau,\mathbb{C})=2 $ if $Y_{\tilde{\tau}}$ is a singleton. On the Galois side, $$F(\phi_\tau )=\{\omega_{E/F}\oplus As_{E/F}(\phi_1),\omega_{E/F}\oplus\omega_{E/F} As_{E/F}(\phi_1) \}. $$
			Moreover, $|\Pi_{\phi_{\tau}}|=2$ and $S_{\phi}=\mu_2\cong S_{\phi_{\tau} }$. Only one member inside the $L$-packet $\Pi_{\phi_{\tau}}$ is $\Sp_4(F)$-distinguished,
			the other one is nongeneric and so it is not $\Sp_4(F)$-distinguished.

			If $Y_{\tilde{\tau}}=\langle\chi_E\rangle$, then $\chi_E^\sigma\in Y_{\tilde{\tau}}$ and so $\chi_E^\sigma=\chi_E$. Then
				\[\dim\Hom_{\Sp_4(F)}(\tau,\mathbb{C})	=\begin{cases}
			4,&\mbox{ if }\chi_E|_{F^\times}=\mathbf{1};\\
			1,&\mbox{ if }\chi_E|_{F^\times}=\omega_{E/F}.\\
			%	2,&\mbox{ if }|Y_{\tilde{\tau}}|=8,|Z_{\tilde{\tau}}|=2;\\
%			2,&\mbox{ if }\chi_E|_{F^\times}\mbox{ is neither }\mathbf{1} \mbox{ nor }\omega_{E/F}.
			\end{cases} \]
			On the Galois side, $F(\phi_{\tau})$ is given by
			 $$\begin{cases}
			\bigcup_{z\in\{0,1\}}	\{\omega_{E/F}\oplus\omega_{E/F}^z \nu_i As_{E/F}(\phi_1),\nu_i\circ N_{E/F}\in Z_{\tilde{\tau}} \},&\mbox{ if }\chi_E|_{F^\times}=\mathbf{1};\\
				\{\omega_{E/F}\oplus As_{E/F}(\phi_1) \},&\mbox{ if }\chi_E|_{F^\times}=\omega_{E/F}.
			\end{cases}
		$$ 
		When $\chi_E|_{F^\times}=\omega_{E/F}$, $S_\phi=\mu_2\hookrightarrow S_{\phi_{\tau}}=\mu_2\times\mu_2$ is the diagonal embedding. There are two representations in $\Pi_{\phi_{\tau}}$ distinguished by $\Sp_4(F)$ and the rest two representations correspond to the nongeneric tempered representations of $\Sp_4(E)$, which are not $\Sp_4(F)$-distinguished.
			\item If $\phi_2=\mathbb{C}\oplus\det\phi_1$, then $\phi_1:WD_E\longrightarrow\GL_2(\mathbb{C})$ is conjugate-orthogonal.
			
			If $\det\phi_1=\mathbf{1}$ or $\det\phi_1$ is not a quadratic character, then $$\dim\Hom_{\Sp_4(F)}(\tau,\mathbb{C}) =1.
			%\begin{cases}
			%2&\mbox{ if }\chi^\sigma\chi=\det\phi_1;\\
		%	1&\mbox{ if }\chi|_{F^\times}=\mathbf{1}\neq\det\phi_1.
		%	\end{cases}
		$$
			 If $\det\phi_1$ is a nontrivial quadratic character, then
			\[\dim\Hom_{\Sp_4(F)}(\tau,\mathbb{C})=\begin{cases}
			1,&\mbox{ if } |Y_{\tilde{\tau}}|=1;\\
			3,&\mbox{ if }|Y_{\tilde{\tau}}|=2.
			\end{cases}  \]
			On the Galois side, there exists a character $\nu$ such that $\det\phi_1=\nu^\sigma\nu^{-1}$. So $$(\phi_1\otimes\nu)^\ell=\phi_1^\vee\otimes\nu^\sigma=\phi_1\otimes\nu.$$  There exist $\tilde{\rho}_1$ and $\tilde{\rho}_2$ such that $\tilde{\rho}_i|_{WD_E}=\phi_i\otimes\nu$ and $\det\tilde{\rho}_i=\omega_{E/F}\nu|_{F^\times}$. Then
			\[F(\phi_\tau )=\{\mathbb{C}\oplus \tilde{\rho}_1^\vee\otimes\tilde{\rho}_2\} \]
			if $|Y_{\tilde{\tau}}|=1$. If $Y_{\tilde{\tau}}=\langle\det\phi_1\rangle $, then $\phi_1^\ell=\phi_1^\vee=\phi_1$. There exists a parameter $\tilde{\rho}_3$ such that $\tilde{\rho}_3|_{WD_E}=\phi_1$. Therefore
			\[F(\phi_\tau )= \{\mathbb{C}\oplus \tilde{\rho}_1^\vee\otimes
			\tilde{\rho}_2,\mathbb{C}\oplus(\tilde{\rho}_3\oplus\tilde{\rho}_3^\vee),\mathbb{C}\oplus(\tilde{\rho}_3\oplus\tilde{\rho}_3^\vee)\omega_{E/F} \}. \]
		\end{itemize}
%	\end{enumerate}
\item 
Suppose that $\phi_1=\chi\chi_1\oplus\chi\chi_2$ is reducible and $\phi_{\tilde{\tau}}=\chi(\chi_1\chi_2\oplus\chi_1\oplus\chi_2\oplus\mathbb{C})$. Then $$\phi_\tau=\chi_1\oplus\chi_2\oplus\mathbb{C}\oplus\chi_2^{-1}\oplus\chi_1^{-1}\mbox{ with }(\chi^2\chi_1\chi_2)|_{F^\times}=\mathbf{1} .$$
%\begin{enumerate}[(a)]
	
	\noindent(d). 
If $\chi^\sigma\chi\chi_1\chi_2=\mathbf{1}=\chi^\sigma\chi_1^\sigma\chi\chi_2$, then $\chi^\sigma_1=\chi_1$ and $\chi_2^\sigma=\chi_2$. 

\noindent(d1) Assume $\chi_1\neq\chi_2^{\pm1}$. Then
%Suppose that $\phi_1\neq\phi_2$.
\begin{itemize}
\item
	 If $|Y_{\tilde{\tau}} |=|Z_{\tilde{\tau}}|=r$ where $r=1,2$ or $4$, then $$\dim\Hom_{\Sp_4(F)}(\tau,\mathbb{C})=\begin{cases}
	4,&\mbox{ if }r=1;\\6,&\mbox{ if }r=2;\\ 8,&\mbox{ if }r=4.
	\end{cases}$$ On the Galois side, there exist two characters $\chi_F$ and $\chi_F'$ of $F^\times$ such that $\chi_1=\chi_F\circ N_{E/F}$ and $\chi_2=\chi_F'\circ N_{E/F}$. Moreover, $\chi_F^2\neq\omega_{E/F}$ and $\chi_F'^2\neq\omega_{E/F}$. So 
	if $r=1$,
	\[F(\phi_{\tau})=\bigcup_{z_1,z_2\in\{0,1\}}\{\chi_F\omega_{E/F}^{z_1}\oplus\chi_F'\omega_{E/F}^{z_2}\oplus\mathbb{C}\oplus(\chi_F\omega_{E/F}^{z_1})^{-1}\oplus(\chi_F'\omega_{E/F}^{z_2})^{-1} \} .\]
	If $r=2$, then the extra $2$ lifted parameters are given by
	\[\chi_F\omega_{E/F}^z\oplus \chi_F'\oplus\omega_{E/F}\oplus\chi_F'\omega_{E/F}\oplus \chi_F^{-1}\omega_{E/F}^z,z\in\{0,1 \} \]
	with $(\chi_F')^2=\mathbf{1}$. If $r=4$, another $2$ parameter lifts are given by
	\[ \chi_F'\omega_{E/F}^z\oplus \chi_F\oplus\omega_{E/F}\oplus\chi_F\omega_{E/F}\oplus \chi_F'\omega_{E/F}^z,z\in\{0,1\}. \]
	\item If $|Y_{\tilde{\tau}}|=4$ and $Z_{\tilde{\tau}}=\langle\chi_1\rangle $,  then $\dim\Hom_{\Sp_4(F)}(\tau,\mathbb{C})=3$. On the Galois side, $\chi_F^2=\mathbf{1}$ and $(\chi_F')^2=\omega_{E/F}$. There are $3$  elements in $F(\phi_{\tau})$. For each $\tilde{\phi}_i\in F(\phi_{\tau})$, one has $\deg\Phi(\tilde{\phi}_i)=2$ and $d_0(\tilde{\phi}_i)=2$.
	\item If $|Y_{\tilde{\tau}}|=2$ and $|Z_{\tilde{\tau}}|=1$, then $\dim\Hom_{\Sp_4(F)}(\tau,\mathbb{C})=2$ and $|F(\phi_\tau  )|=2$.
	%and
	%\[\deg\Phi (\phi_i)=2 \]
%$\phi_i=\chi_F\omega_{E/F}^i\oplus\chi_F'\oplus\mathbb{C}\oplus\chi_F^{-1}\omega_{E/F}^i\oplus\chi_F'^{-1}$.
	%\item If $|Y_{\tilde{\tau}}|=1$, then
	%$\dim\Hom_{\Sp_4(F)}(\tau,\mathbb{C})=4 $.
\end{itemize}
\noindent(d2) If $\chi_1=\chi_2=\chi_F\circ N_{E/F}$, then 
%If $\chi_1=\mathbf{1}$, i.e. $\phi_1=\phi_2$, then
\[\dim\Hom_{\Sp_4(F)}(\tau,\mathbb{C})=\begin{cases}
6,&\mbox{ if }\chi_F^2=\mathbf{1};\\
2,&\mbox{ if }\chi_F^2=\omega_{E/F};\\
4,&\mbox{ otherwise. }
\end{cases}  \]
%\noindent(d3)
If  $\chi_F^2=\mathbf{1}$, then $Y_{\tilde{\tau}}=Z_{\tilde{\tau}}=\langle\chi_1\rangle$ and $S_{\phi_\tau }=\mu_2$. There are $5$ elements in $F(\phi_{\tau})$ and at the point $$\phi_1=2\chi_F\oplus 2\chi_F\omega_{E/F}\oplus\mathbb{C},$$
we have $\deg \Phi(\phi_1)=2$. The rest $4$ parameters in $F(\phi_\tau )$
are of degree $1$ and are given by $4\chi_F\omega_{E/F}^{z}\oplus\mathbb{C}$ and $2\chi_F\omega_{E/F}^{z'}\oplus\chi_F\oplus\chi_F\omega_{E/F}\oplus\omega_{E/F}$ with $z,z'\in\{ 0,1\}$.
\par 
If $|Y_{\tilde{\tau}}|=\langle\chi_1\rangle$ and $|Z_{\tilde{\tau}}|=1$, then $F(\phi_\tau )=\{\phi=2\chi_F\oplus 2\chi_F^{-1}\oplus\mathbb{C} \}$ is a singleton, where $\chi_F^2=\omega_{E/F}$. There is a group embedding 
\[S_{\phi}=\mathbf{1}\hookrightarrow S_{\phi_\tau }=\mu_2. \]
 The degree $\deg\Phi(\phi)=4$ and $d_0(\phi)=2$.
\par
If $|Y_{\tilde{\tau}}|=1$, then $\dim\Hom_{\Sp_4(F)}(\tau,\mathbb{C}) =4$ and
$$F(\phi_\tau )=\{\phi=\chi_F\oplus\chi_F\omega_{E/F}\oplus\mathbb{C}\oplus\chi_F^{-1}\omega_{E/F}\oplus\chi_F^{-1}\}\cup\bigcup_{z\in\{0,1\}} \{2 \chi_F\omega_{E/F}^{z}\oplus\mathbb{C} \oplus2\chi_F^{-1}\omega_{E/F}^{z} \}$$ with $\deg\Phi (\phi)=2$ and the other two parameters are of degree $1$.

\noindent(e).
If $\chi^\sigma\chi_1^\sigma\chi\chi_2=\mathbf{1}=\chi|_{F^\times}$ and $\chi_1\chi_2\neq\mathbf{1}$, then $\chi_1^\sigma=\chi_2^{-1}\neq\chi_1$.
\begin{itemize}
	\item  If $\chi_1^2=\mathbf{1}$, then $|Y_{\tilde{\tau}}|=4$ and $|Z_{\tilde{\tau}}|=2$. So
	$X_{\tilde{\tau}}=\{\mathbf{1},\chi_1|_{F^\times} \}$ and $$\dim\Hom_{\Sp_4(F)}(\tau,\mathbb{C})=\frac{2}{4/2}=1. $$
	On the Galois side, set $\rho=Ind_{W_E}^{W_F}\chi_1$. Then
	\[F(\phi_\tau )=\{\rho\oplus\mathbb{C}\oplus\rho^\vee\}. \]
	% $$F(\phi_\tau )=\{Ind_{W_E}^{W_F}\chi_1\oplus Ind_{W_E}^{W_F}\chi_2\oplus\mathbb{C} \}.$$
	\item If $\chi_1^2\neq\mathbf{1}$, then $\chi_2^2\neq\mathbf{1}$. If $\chi_1\chi_2$ is a quadratic character, then $\chi_1\chi_2=\chi_F\circ N_{E/F}$ with $\chi_F^2=\mathbf{1}$. So $Y_{\tilde{\tau}}=Z_{\tilde{\tau}}$ and $\dim\Hom_{\Sp_4(F)}(\tau,\mathbb{C})=1$. 
	% On the Galois side,
	%\[F(\phi_{\tau} )= \]
	Suppose that $\chi_1\chi_2$ is not  a quadratic character. Then $|Y_{\tilde{\tau}}|=1$ and
	\[\dim\Hom_{\Sp_4(F)}(\tau,\mathbb{C})=1.  \]
\end{itemize}
	On the Galois side,
\[F(\phi_{\tau})=\{\mathbb{C}\oplus \tilde{\rho}_1^\vee\otimes\tilde{\rho}_2\} \]
where $\tilde{\rho}_i|_{WD_E}=\phi_i\otimes\nu$ with $\det\tilde{\rho}_i=\omega_{E/F}\cdot \nu|_{F^\times}$ where $\det\phi_1=\nu^\sigma\nu^{-1}$.

\noindent(f).
If $\chi|_{F^\times}=\mathbf{1}=\chi_i|_{F^\times}$, $\chi_1\neq\chi_2^{\pm1
}$, then $Y_{\tilde{\tau}}=Z_{\tilde{\tau}}$ and $$\dim\Hom_{\Sp_4(F)}(\tau,\mathbb{C} )=\begin{cases}
8,&\mbox{ if }|Y_{\tilde{\tau}}|=4;\\
3,&\mbox{ if }|Y_{\tilde{\tau}}|=2;\\
1,&\mbox{ if }|Y_{\tilde{\tau}}|=1.
\end{cases} $$
On the Galois side, if $|Y_{\tilde{\tau}}|=4$, there are $6$ elements in $X_{\tilde{\tau}}$ which coincides with the case (d1). 

If $Y_{\tilde{\tau}}=\langle\chi_1\rangle=\langle\chi_F\circ N_{E/F}\rangle$, then
\[F(\phi_{\tau} )=\{\chi_F\oplus\omega_{E/F}\oplus\chi_F\omega_{E/F}\oplus Ind_{W_E}^{W_F}\chi_2\}\cup\bigcup_{z\in\{0,1\} }\{  Ind_{W_E}^{W_F}\chi_2\oplus2\chi_F\omega_{E/F}^z\oplus\mathbb{C} \} \]
with $\chi_F^2=\mathbf{1}$.
%\[F(\phi_\tau )=\omega_{E/F}\oplus Ind_{W_E}^{W_F}\chi_1\oplus Ind_{W_E}^{W_F}\chi_2. \]

\noindent(g).
If $\phi_1^\ell=\phi_2^\vee\neq\phi_1^\vee$, then $\chi_1^\sigma=\chi_1=(\chi\chi^\sigma)^{-1}$ and $\chi_2^\sigma\chi_2=\mathbf{1},\chi_2\neq\chi_2^\sigma$.  So $\chi_2|_{F^\times}=\mathbf{1}$.
\begin{itemize}
	\item If $Y_{\tilde{\tau}}=Z_{\tilde{\tau}}$, then
	$\dim\Hom_{\Sp_4(F)}(\tau,\mathbb{C})= 2$. On the Galois side, $\chi_1=\chi_F\circ N_{E/F}$ and
	\[F(\phi_{\tau})=\bigcup_{z\in\{0,1\} }\{\chi_F\omega_{E/F}^z\oplus\chi_F^{-1}\omega_{E/F}^z\oplus\omega_{E/F}\oplus Ind_{W_E}^{W_F}\chi_2  \}. \]
	\item If $|Y_{\tilde{\tau}} |=2|Z_{\tilde{\tau}}|$, then $\dim\Hom_{\Sp_4(F)}(\tau,\mathbb{C})=1$. In this case, $\chi_1=\chi_F\circ N_{E/F}$ with $\chi_F^2=\omega_{E/F}$ and $F(\phi_\tau )=\{\chi_F\oplus\chi_F\omega_{E/F}\oplus\omega_{E/F}\oplus Ind_{W_E}^{W_F}\chi_2 \}$. 
\end{itemize}

\noindent(h).
 If $\chi_1=\mathbf{1}=\chi_2|_{F^\times}$ and $\chi|_{F^\times}=\omega_{E/F}$, then $\phi_\tau $ coincides with the case (d2) if $\chi_2=\mathbf{1}$, or (f) if $\chi_2\neq\mathbf{1}$.
\end{enumerate}	
\end{enumerate}
%\end{enumerate}
Then we have finished the proof of Theorem \ref{sp(4)}. 
\end{proof}

\begin{rem}
	Due to \cite[Theorem 4.2]{L2018pacific}, the non-generic tempered representation $\tau$ of $\Sp_4(E)$ can never be $\Sp_4(F)$-distinguished. If the representation $\tau$ is square-integrable and the $L$-packet $\Pi_{\phi_\tau}$
is a singleton, then the Prasad conjecture for $\Sp_4$ holds, which has been shown in \cite{L2018pacific} using the local theta correspondence. 
\end{rem}
\bibliographystyle{amsalpha}
\bibliography{SO(4)}

\providecommand{\bysame}{\leavevmode\hbox to3em{\hrulefill}\thinspace}
\providecommand{\MR}{\relax\ifhmode\unskip\space\fi MR }
% \MRhref is called by the amsart/book/proc definition of \MR.
\providecommand{\MRhref}[2]{%
  \href{http://www.ams.org/mathscinet-getitem?mr=#1}{#2}
}
\providecommand{\href}[2]{#2}
\begin{thebibliography}{GGP11}

\bibitem[AM17]{matringe2017test}
U.~K. Anandavardhanan and Nadir Matringe, \emph{Test vectors for local
  periods}, Forum Math. \textbf{29} (2017), no.~6, 1245--1260. \MR{3719298}

\bibitem[AP03]{anandavardhanan2003distinguished}
UK~Anandavardhanan and Dipendra Prasad, \emph{Distinguished representations for
  {$\rm SL (2)$}}, Math. Res. Lett. \textbf{10} (2003), no.~6, 867--878.

\bibitem[AP18]{anandavardhanan2016distinguished}
\bysame, \emph{Distinguished representations for $\rm{SL(n)}$}, Math. Res.
  Lett. \textbf{25} (2018), no.~6, 1695--1717.

\bibitem[BP18]{beuzart2017distinguished}
Rapha{\"e}l Beuzart-Plessis, \emph{On distinguished square-integrable
  representations for {G}alois pairs and a conjecture of {P}rasad}, Invent.
  Math. \textbf{214} (2018), no.~1, 437--521.

\bibitem[FLO12]{lapid2012unitary}
Brooke Feigon, Erez Lapid, and Omer Offen, \emph{On representations
  distinguished by unitary groups}, Publ. Math. Inst. Hautes \'Etudes Sci.
  \textbf{115} (2012), 185--323. \MR{2930996}

\bibitem[GGP11]{gan2011symplectic}
Wee~Teck Gan, Benedict~H. Gross, and Dipendra Prasad, \emph{Symplectic local
  root numbers, central critical {L}-values, and restriction problems in the
  representation theory of classical groups}, Ast{\'e}risque (2011), no.~346,
  1--109.

\bibitem[GT10]{takeda2010Sp(4)}
Wee~Teck Gan and Shuichiro Takeda, \emph{The local {L}anglands conjecture for
  $\rm{Sp(4)}$}, Int. Math. Res. Not. IMRN (2010), no.~15, 2987--3038.
  \MR{2673717}

\bibitem[Lu17]{lu2016new}
Hengfei Lu, \emph{$\rm{GSp(4)}$ period problems over a quadratic field
  extension}, Ph.D. thesis, NUS, 2017.

\bibitem[Lu18a]{lu2018GSp(4)}
\bysame, \emph{The {P}rasad conjecture for $\rm{GSp(4)}$ and its inner form},
  arxiv preprint arxiv:1802.10336v1 (2018).

\bibitem[Lu18b]{L2018pacific}
\bysame, \emph{Theta correspondence and the {P}rasad conjecture for $\rm
  {SL(2)}$}, Pac. J. Math. \textbf{295} (2018), no.~2, 477--498.

\bibitem[Pra92]{dipendra1992invariant}
Dipendra Prasad, \emph{Invariant forms for representations of {${\rm GL}_2$}
  over a local field}, Amer. J. Math. \textbf{114} (1992), no.~6, 1317--1363.
  \MR{1198305}

\bibitem[Pra15]{prasad2015arelative}
\bysame, \emph{A 'relative' local {L}anglands correspondence}, arXiv preprint
  arXiv:1512.04347 (2015).

\bibitem[SV17]{sakellaridis2012periods}
Yiannis Sakellaridis and Akshay Venkatesh, \emph{Periods and harmonic analysis
  on spherical varieties}, Ast{\'e}risque \textbf{396} (2017).

\end{thebibliography}
\end{document}